\documentclass[final,leqno,letterpaper]{etna}

\usepackage{graphicx}
\usepackage{amssymb}
\usepackage{epstopdf}
\usepackage{cite}
\usepackage{rotfig_tikz,pgfplots}

\usepackage{tabularx}
\usepackage{booktabs} 
\newcolumntype{C}[1]{>{\centering\arraybackslash}p{#1}} 
\usepackage{multirow}
\newtheorem{remark}[theorem]{Remark}

\newcommand{\cplxs}{\mbox{$\mathbb{C}$}}
\newcommand{\cn}{\mbox{$\cplxs^{n}$}}
\newcommand{\cnxn}{\mbox{$\cplxs^{n \times n}$}}
\newcommand{\absval}[1]{\mbox{$\mid\!#1\!\mid$}}
\newcommand{\norm}[1]{\mbox{$\parallel\!#1\!\parallel$}}

\newcommand{\spn}[1]{\mbox{{\rm span}$\left\{#1\right\}$}}
\newcommand{\fl}[1]{\mbox{{\rm fl}$(#1)$}}
\newcommand{\eee}{\mbox{$\mathcal{E}$}}
\newcommand{\kay}{\mbox{$\mathcal{K}$}}
\newcommand{\ehl}{\mbox{$\mathcal{L}$}}
\newcommand{\cue}{\mbox{$\mathcal{Q}$}}
\newcommand{\ess}{\mbox{$\mathcal{S}$}}
\newcommand{\you}{\mbox{$\mathcal{U}$}}
\newcommand{\zee}{\mbox{$\mathcal{Z}$}}

\title{On pole-swapping algorithms \\ for the eigenvalue problem\thanks{This research was partially supported by 
the Research Council KU Leuven, project 
C14/16/056 (Inverse-free Rational Krylov Methods: Theory and Applications).}}
\author{Daan Camps\footnotemark[2]\and 
  Thomas Mach\footnotemark[3]\and
  Raf Vandebril\footnotemark[4]\and
  David~S.~Watkins\footnotemark[5]}

\shorttitle{Pole-swapping algorithms} 
\shortauthor{D.~Camps, T.~Mach, R.~Vandebril, and D.~S.~Watkins}

\begin{document}

\maketitle

\renewcommand{\thefootnote}{\fnsymbol{footnote}}

\footnotetext[2]{Computational Research Division, Lawrence Berkeley National Laboratory, California
  (\texttt{dcamps@lbl.gov}).}%
\footnotetext[3]{Department of Mathematical Sciences, Kent State University, Ohio
  (\texttt{tmach1@kent.edu}).}%
\footnotetext[4]{Department of Computer Science, KU Leuven,
  Belgium (\texttt{raf.vandebril@cs.kuleuven.be}).}
\footnotetext[5]{Department of Mathematics, Washington State University
   (\texttt{watkins@math.wsu.edu})}

\renewcommand{\thefootnote}{\arabic{footnote}}

\begin{abstract}
Pole-swapping algorithms, which are generalizations of the QZ algorithm
for the generalized eigenvalue problem, are studied.  A new modular 
(and therefore more flexible) convergence theory that applies to all pole-swapping
algorithms is developed.  A key component of all such algorithms is a procedure that swaps two
adjacent eigenvalues in a triangular pencil.  An improved swapping routine is developed,
and its superiority over existing methods is demonstrated by a backward error analysis and 
numerical tests.  The modularity of the new convergence theory and the generality of the
pole-swapping approach shed new light on bi-directional chasing algorithms, optimally
packed shifts, and bulge pencils,  and allow the design of novel algorithms.
\end{abstract}

\begin{keywords}
eigenvalue, QZ algorithm, pole swapping, convergence
\end{keywords}

\begin{AMS}
  65F15, 
  15A18 
\end{AMS}

\section{Introduction}

The standard algorithm for computing the eigenvalues of a small to medium-sized non-Hermitian matrix
$A\in\cnxn$ is still Francis's implicitly-shifted QR algorithm \cite{Fra61b,Wat11}.  In many applications, eigenvalue
problems arise naturally as generalized eigenvalue problems for a pencil $A - \lambda B$, and
for these problems the Moler-Stewart variant of Francis's algorithm \cite{MolSte73}, commonly called the QZ 
algorithm,  can be used.  
In this paper we may refer sometimes to a pencil $A - \lambda B$ and other times to a pair $(A,B)$.  
Either way, we are talking about the same object.  

A few years ago we published a generalization of the QZ algorithm \cite{VanWat12g}.  
More recently an even more general algorithm, the \emph{rational QZ (RQZ) algorithm}, was presented by 
Camps, Meerbergen, and Vandebril \cite{CaMeVa19a}.  This arose from the study of rational 
Arnoldi methods and is related to work of Berljafa and G\"uttel \cite{BerGut15}.  

In this paper we discuss the RQZ algorithm and introduce several
variants.  We develop a new modular (and therefore more flexible)
convergence theory that can be applied immediately to all variants.

We reinterpret the QZ algorithm and show that it can be viewed as a 
pole-swapping algorithm with poles at infinity. Moreover we will 
show that the algorithm \cite{KaKrLa14}
for optimally packed chains of bulges is a disguised implementation of pole swapping. 

A key component of the RQZ and related algorithms is a procedure that swaps two
adjacent eigenvalues in a triangular pencil.  We present an improved swapping routine and 
demonstrate its superiority by numerical experiments and a backward error analysis.


Double-shift pole-swapping algorithms that can be applied to real matrix pencils 
exist \cite{Cam19,CaMeVa19b}.   All of what is 
discussed in this paper for single shifts can be extended to the double-shift case, but we have not worked out 
every detail. The one item that will require further thought is the extension of the improved swapping
routine of Section~\ref{sec:newswap} to blocks larger than $1\times 1$.
A significant advantage of sticking to the complex single-shift case, as we have done here, 
is simplicity and clarity of presentation.  

\section{Hessenberg pairs}\label{sec:hespairs}

A pencil $A - \lambda B$ is called a \emph{regular pencil} or \emph{regular pair} 
if there is at least one complex $\mu$ such that $A - \mu B$ is invertible.  
Throughout this paper we make the blanket assumption of regularity.  

A matrix $A\in\cnxn$ is in (upper) \emph{Hessenberg form} if every entry below the first subdiagonal is zero.  It is 
in \emph{proper} Hessenberg form if every subdiagonal entry is nonzero, i.e.\ $a_{j+1,j} \neq 0$ for $j=1$, \ldots, $n-1$.  
A preliminary step for the $QZ$ algorithm is to reduce the pair $(A,B)$ to Hessenberg-triangular form.  That is, 
$(A,B)$ is transformed by a unitary equivalence to a new pair $(\check{A},\check{B})$ for which $\check{A}$ is 
upper Hessenberg and $\check{B}$ is upper triangular.  Notice that if $\check{A}$ is not properly Hessenberg,
the eigenvalue problem can be split immediately into two or more independent subproblems.  Thus we can always
assume that we are dealing with a matrix in proper Hessenberg form.  

In the new theory we deal with a more general class of Hessenberg pencils.  The pair $(A,B)$ is called a 
\emph{Hessenberg pair} if both $A$ and $B$ are Hessenberg matrices.   
If $a_{j+1,j} = 0 = b_{j+1,j}$ for some $j$, we
can immediately split the eigenvalue problem into two smaller problems.  We therefore eliminate that case
from further consideration.  
For reasons that will become apparent later, the ratios $a_{j+1,j}/b_{j+1,j}$, $j=1$, \ldots, $n-1$ 
are called the \emph{poles} of the  Hessenberg pair  $(A,B)$.   
In the case $b_{j+1,j} = 0$, we have an infinite pole.   The Hessenberg-triangular form 
is a special Hessenberg pair for which all of the poles are infinite.  

Closely related to $(A,B)$ is the \emph{pole pair} $(A_{\pi},B_{\pi})$  (or \emph{pole pencil}
$A_{\pi}-\lambda B_{\pi}$) obtained from $(A,B)$ by deleting the first
row and last column.  The pole pencil is upper triangular, and its eigenvalues are
obviously the poles of $(A,B)$.

\subsection*{Operations on Hessenberg pairs}

Introducing terminology that we have used in some of our recent work 
\cite{AuMaRoVaWa18,AuMaRoVaWa18g,AuMaRoVaWa19,AuMaVaWa15},  we define a \emph{core transformation}  
(or \emph{core} for short) to be a unitary matrix that acts only on two adjacent rows/columns, for example,
\begin{displaymath}
Q_{3} = \left[\begin{array}{ccccc}
1 & & & & \\ & 1 & & & \\ & & {*} & {*} & \\  & & {*} & {*} & \\  & & & & 1
\end{array}\right],
\end{displaymath}
where the four asterisks form a $2 \times 2$ unitary matrix.  Givens rotations are examples of core transformations.
Our core transformations always have subscripts that tell where the action is:  $Q_{j}$ acts on rows/columns
$j$ and $j+1$.

Following \cite{CaMeVa19a} we introduce two types of operations, or \emph{moves}, both of which manipulate the poles
in the pair.  Let $\sigma_{1} = a_{21}/b_{21}$, \ldots, $\sigma_{n-1} = a_{n,n-1}/b_{n,n-1}$ denote the 
poles of the  Hessenberg pair $(A,B)$.  

\subsection*{Changing a pole at the top or bottom.  (Type I move)}  
We can change the pole $\sigma_{1}$ to any value we want by applying a core transformation
$Q_{1}^{*}$ to the pencil on the left.   Suppose we want to change $\sigma_{1}$ to $\rho$, say.  Noting that only the 
first two entries of $(A - \rho B)e_{1}$ can be nonzero, we deduce that there is a $Q_{1}$ such that  the second
entry of $Q_{1}^{*}(A - \rho B)e_{1}$ is zero.  In other words, 
\begin{equation}\label{eq:qzstart}
Q_{1}^{*}(A - \rho B)e_{1} = \gamma e_{1} 
\end{equation}
for some 
$\gamma$.   If we then define 
$\hat{A} = Q_{1}^{*}A$ and $\hat{B} = Q_{1}^{*}B$, then $(\hat{A} - \rho \hat{B})e_{1} = \gamma e_{1}$, which
implies that $\hat{a}_{21} - \rho \hat{b}_{21} = 0$.  This means that $\rho = \hat{a}_{21}/\hat{b}_{21}$ is the new
first pole of $(\hat{A},\hat{B})$.  The other poles remain fixed, as they are untouched by the transformation.  

This operation fails only if $\hat{a}_{21} = 0 = \hat{b}_{21}$, yielding $\rho = 0/0$.   This happens exactly when 
the first columns of $A$ and $B$ are proportional.   But this is not such a failure after all, as it exposes 
$\hat{a}_{11}/\hat{b}_{11}$  as an eigenvalue of the pencil and allows us to deflate to a smaller problem by deleting 
the first row and column.  

In summary, if we want to replace pole $\sigma_{1}$ by $\rho$, we will either succeed in doing so or get a deflation 
of an eigenvalue.

\begin{remark}
When we write something like $A - \rho B$ here and elsewhere, this should be viewed as shorthand for
 $\beta A - \alpha B$ where $\alpha$ and $\beta$ are any scalars for which $\rho = \alpha/\beta$.   As a practical
 matter this allows us to use modest sized $\alpha$ and $\beta$ even when $\rho$ is very 
 large, and in particular it allows us to implement the case $\rho = \infty$ by taking $\beta = 0$.
 \end{remark}
  
The pole $\sigma_{n-1}$ at the bottom can also be replaced by any other pole, say $\tau$, by a similar procedure.
We want to  transform the pencil $A - \lambda B$ to $\hat{A} - \lambda \hat{B} = (A - \lambda B)Z_{n-1}$ with 
$\hat{a}_{n,n-1}/\hat{b}_{n,n-1} = \tau$.   Noting that the row vector $e_{n}^{T}(A - \tau B)$ has nonzero entries only in its
last two positions, we see that there must be a core transformation $Z_{n-1}$ that maps it to a multiple of $e_{n}^{T}$, 
i.e.\ $e_{n}^{T}(A - \tau B)Z_{n-1} = \gamma e_{n}^{T}$ for some 
$\gamma$.   This is the desired transformation, 
since it implies $e_{n}^{T}(\hat{A} - \tau \hat{B}) = \gamma e_{n}^{T}$, which is equivalent to 
$\hat{a}_{n,n-1}/\hat{b}_{n,n-1} = \tau$.

This fails only if $\hat{a}_{n,n-1} = 0 = \hat{b}_{n,n-1}$, yielding $\tau = 0/0$, which happens exactly when the 
$n$th rows of $A$ and $B$ are proportional.  But again this is not really 
a failure at all, since it allows $\hat{a}_{nn}/\hat{b}_{nn}$ to be extracted as an eigenvalue and the problem to be
deflated to a smaller one.  

This discussion helps motivate the following definition.  
A Hessenberg pair is called a \emph{proper Hessenberg pair} if
three conditions hold:  (i) $\absval{a_{j+1,j}} + \absval{b_{j+1,j}} > 0$ for $j=1$, \ldots, $n-1$,  (ii)  the first columns 
of $A$ and $B$ are not proportional, (iii) the last rows of $A$ and $B$ are not proportional.  The first condition
just says that for each $j$, at least one of $a_{j+1,j}$ and $b_{j+1,j}$ is nonzero.  If this condition is 
not satisfied, we can immediately reduce the pencil to two smaller pencils.   If either of conditions (ii) 
and (iii) is not satisfied, we can also reduce the problem, as we know from the discussion immediately above.
Therefore, we can always assume, without loss of generality, that we are working with a proper Hessenberg pair.  

\begin{proposition}\cite{CaMeVa19a}\quad\label{prop:firstcol}
 In a proper Hessenberg pair, the core transformation $Q_{1}$ that replaces pole $\sigma_{1}$ by $\rho$ satisfies
 \begin{displaymath}
 Q_{1}e_{1} =  \delta\,(A - \rho B)(A - \sigma_{1}B)^{-1}e_{1}
 \end{displaymath}
 for some nonzero $\delta$.   
 \end{proposition}    
 
 \begin{proof}
 From our construction we have $Q_{1}e_{1} = \gamma^{-1}(A - \rho B)e_{1}$.  
  Since $\sigma_{1}$ is the first 
 pole of the pair $(A,B)$, we have $(A - \sigma_{1}B)e_{1} = \check{\gamma}e_{1}$ for some $\check{\gamma}$.
 The properness assumption guarantees that both $\gamma$ and $\check{\gamma}$ are nonzero.  
 Therefore $Q_{1}e_{1} = \delta (A - \rho B)(A - \sigma_{1}B)^{-1}e_{1}$, where $\delta = (\gamma\check{\gamma})^{-1}$.   
 \hfill\end{proof}
 
 \begin{remark}
 The insertion of the extra factor $(A - \sigma_{1}B)^{-1}$ may seem mysterious.  
 As we shall see later, this is just what is needed
 for a consistent convergence theory.   In the product $(A - \rho B)  (A - \sigma_{1}B)^{-1}$, the factor $A - \rho B$ signals
 that the pole $\rho$ is entering the pencil, while the factor $(A - \sigma_{1}B)^{-1}$ signals that the pole $\sigma_{1}$ is 
 leaving.  
 \end{remark}
 
 \begin{proposition}\cite{CaMeVa19a}\quad\label{prop:lastrow}
 In a proper Hessenberg pair, the core transformation $Z_{n-1}$ that replaces pole $\sigma_{n-1}$ by $\tau$ satisfies
 \begin{displaymath}
 e_{n}^{T}Z_{n-1}^{*} =  \delta\, e_{n}^{T}(A - \sigma_{n-1}B)^{-1}(A - \tau B)
 \end{displaymath}
 for some nonzero $\delta$.   
 \end{proposition}    
 
 \begin{proof}
 From our construction we have $e_{n}^{T}Z_{n-1}^{*} = \gamma^{-1}e_{n}^{T}(A - \tau B)$.  Since $\sigma_{n-1}$ is the 
 last  pole of the pair $(A,B)$, we have $e_{n}^{T}(A - \sigma_{n-1}B) = \check{\gamma}e_{n}^{T}$ 
 for some nonzero $\check{\gamma}$.   Therefore $e_{n}^{T}Z_{n-1}^{*} = 
 \delta e_{n}^{T}(A - \sigma_{n-1}B)^{-1}(A - \tau B)$, where $\delta = (\gamma\check{\gamma})^{-1}$.   
 \hfill\end{proof}
 
 The arithmetic cost of a move of type I is just the cost of multiplying $A$ and $B$ by a single core transformation,
 $Q_{1}^{*}$ or $Z_{n-1}$.  If the cores are Givens rotations applied in the conventional way, the cost is about 
 $8n$ multiplications
 and $4n$ additions, or $12n$ (complex) flops.  Different implementations could yield slightly different flop counts, 
 but regardless of the details the cost will be $O(n)$.
 
 Standard backward error analysis \cite{Wil65} shows that moves of type I are backward stable.
 
 \subsection*{Interchanging two poles.  (Type II move)} 

The second of the two allowed operations is to interchange two adjacent poles by a unitary equivalence 
$\hat{A} - \lambda \hat{B}  = Q_{j}^{*}(A - \lambda B)Z_{j-1}$.  To understand this, consider the pole pencil
$A_{\pi} - \lambda B_{\pi}$ obtained by discarding the first row and last column from $A - \lambda B$.  This 
pencil is upper triangular and has $\sigma_{1}$, \ldots, $\sigma_{n-1}$ as its eigenvalues.   
There are standard techniques \cite{BaiDem93,KagPor96a,Kagpor96b,VanD81}, \cite[\S\S~4.8,\,6.6]{Wat07} for interchanging 
any two adjacent eigenvalues $\sigma_{j-1}$ and $\sigma_{j}$. We will describe an improved method  
in Section~\ref{sec:newswap}.  Each of these requires only an equivalence transformation 
$\tilde{Q}_{j-1}^{*}(A_{\pi} - \lambda B_{\pi})\tilde{Z}_{j-1}$ by two core transformations $\tilde{Q}_{j-1}$ and
$\tilde{Z}_{j-1}$ of dimension $n-1$.  We then enlarge these matrices by adjoining a row and column
to the top of $\tilde{Q}_{j-1}$ and the bottom of $\tilde{Z}_{j-1}$:
\begin{displaymath}
Q_{j} = \left[\begin{array}{cc} 1 & 0 \\ 0 & \tilde{Q}_{j-1}\end{array}\right] \qquad
Z_{j-1} = \left[\begin{array}{cc}  \tilde{Z}_{j-1} & 0  \\ 0 & 1 \end{array}\right].
\end{displaymath}
Then $\hat{A} - \lambda \hat{B}  = Q_{j}^{*}(A - \lambda B)Z_{j-1}$ is the desired transformation.  

If the swap is done as described by Van Dooren \cite{VanD81}, the procedure always succeeds and is 
backward stable in a sense.  Our new swapping procedure will be shown to have improved stability.
In order not to interrupt the flow of the paper, we defer the description of the new procedure, as well as a  
discussion of backward errors, to Section~\ref{sec:newswap}.   

The flop count for a move of type II is about the same as for a move of type I, namely $12n$ if the core transformations
are implemented as Givens rotations.  In any event, the flop counts for moves of type I and II are about the same, 
and each move costs $O(n)$ 
flops.\footnote{This is the correct count for the case when only eigenvalues are being computed.  If eigenvectors
or some deflating subspaces are wanted as well, the transforming matrices $Q$ and $Z$ also need to be updated
on each move.  This adds about $6n$ (complex) flops for a type I move and $12n$ flops for a type II, but the total is 
still $O(n)$.}

\begin{remark}
We have one type of move that is able to change a pole at one end or the other 
and another type that swaps poles in the middle.  
It is natural to ask whether we can devise a move that changes a single pole in the middle.  The answer is  \emph{no}.
Consider a transformation 
\begin{equation}\label{eq:polestay}
\hat{A} - \lambda \hat{B} = Q^{*}(A - \lambda B)Z,
\end{equation}  
where $Q$ does not touch the first row and $Z$ does not touch the last column.  That is,
\begin{displaymath}
Q = \left[\begin{array}{cc} 1 & \\ & \tilde{Q} \end{array}\right] \quad\mbox{and}\quad
Z  = \left[\begin{array}{cc} \tilde{Z} & \\ & 1\end{array}\right].
\end{displaymath}
Under any such transformation the poles must remain invariant.  This is so because the transformation 
(\ref{eq:polestay}) is equivalent to a transformation $\tilde{Q}^{*}(A_{\pi} - \lambda B_{\pi})\tilde{Z}$
 on the pole pencil.  Since the poles of $A - \lambda B$ are the eigenvalues of the pole pencil, they must 
 remain fixed.
 
 Thus any transformation meant to change a pole must touch either the first row or the last column.
 That's what the moves of type I do.  
\end{remark}

\section{Building an algorithm from the pieces}
\label{sec:build:pieces}

Suppose we want to find the eigenvalues of some regular pair $(A,B)$.
As usual, there are two steps to the process.  The first is a direct method that transforms $(A,B)$ 
to a condensed form, in our case a Hessenberg pencil.  The second step is an iterative process that 
uncovers the eigenvalues of the condensed form.  

In some contexts the reduction phase can be skipped.   As a notable example, the rational Arnoldi process
\cite{BerGut15} applied to a large matrix naturally generates, after $k$ steps, a $k \times k$ Hessenberg pencil.  
The $i$th pole of the pencil is equal to the shift that was used on the $i$th step of the process.   We can obtain
estimates of the eigenvalues of the large matrix by computing the eigenvalues of the pencil.  This requires no
reduction;  we can go directly to the iterative phase.  

\subsection*{Reduction to a Hessenberg pencil}

Moler and Stewart \cite{MolSte73} showed how to reduce $(A,B)$ to Hessenberg-triangular form by a
direct method in $O(n^{3})$ flops.  
The reduction is also described in \cite{GolVan13,Wat07,Wat10} and elsewhere.  
If the resulting pair is not proper, we can split it 
into smaller proper pairs, so let us assume it is proper.   This is a Hessenberg pencil with all poles equal
to $\infty$.  If the user is happy to start from this configuration, s/he can move directly to the iterative phase.

If the user wants to set certain prescribed poles $\sigma_{1}$, \ldots, $\sigma_{n-1}$  
before beginning the iterations, that is also possible.  One obvious procedure is to begin by introducing $\sigma_{n-1}$
at the top of the pencil by a move of type I.  Then $\sigma_{n-1}$ can be swapped with each of the remaining infinite
poles by moves of type II until it arrives at its desired position at the bottom.  The total number of moves is $n-1$.  
Then $\sigma_{n-2}$ can be introduced at the top by a move of type I.  It can then be swapped with each of the remaining
infinite poles until it arrives at its desired position just above $\sigma_{n-1}$.  The total number of moves for this step
is $n-2$.   Then $\sigma_{n-3}$ can be introduced, and so on.  Eventually we get each of $\sigma_{1}$, \ldots,
$\sigma_{n-1}$ into its desired position.  The total number of moves for this phase is about $n^{2}/2$, and the total
flop count is $O(n^{3})$.

One can equally well introduce the poles at the bottom and swap them upward, starting with $\sigma_{1}$, then
$\sigma_{2}$, and so on.  The amount of work is exactly the same, about $n^{2}/2$ moves.  
Better yet, one can take $k \approx (n-1)/2$ and introduce $\sigma_{1}$, \ldots, $\sigma_{k}$ (in reverse order)
at the top and $\sigma_{k+1}$, \ldots, $\sigma_{n-1}$ at the bottom.  This cuts the number of moves in half.  
However one does it, the cost is $O(n^{3})$.

Camps, Meerbergen, and Vandebril \cite{CaMeVa19a} describe a procedure that introduces the poles during
the reduction to Hessenberg form.   They also present an example where a good choice of poles induces a 
deflation in the middle of the pencil.  

\subsection*{The iterative phase (basic algorithm)}

During the discussion of moves of type I in Section~\ref{sec:hespairs} we defined \emph{proper} Hessenberg 
pairs and noted that if a Hessenberg pair is not proper, it can be reduced to smaller pairs that are.
We therefore assume, without loss of generality, that we have a proper Hessenberg pair $(A,B)$ with poles $\sigma_{1}$, \ldots, $\sigma_{n-1}$.  
We now describe an iteration of the RQZ algorithm proposed in \cite{CaMeVa19a}.  
We will call this \emph{the basic algorithm}.

First a shift $\rho$ is chosen.  Any of the usual shifting strategies can be employed here.  The simplest is the
Rayleigh-quotient shift $\rho = a_{nn}/b_{nn}$.  Then $\rho$ is introduced as a pole at the top of the pencil, 
replacing $\sigma_{1}$, by a move of type I.  Next $\rho$ is swapped with $\sigma_{2}$ by a move of type II.  
Then another move of type II is used to swap $\rho$ with $\sigma_{3}$, and so on.  After $n-2$ moves of type II,
$\rho$ arrives at the bottom of the pencil.   The poles are now $\sigma_{2}$, \ldots, $\sigma_{n-1}$, and $\rho$.
Finally a move of type I is used to remove the pole $\rho$ from the bottom, replacing it by a new pole $\sigma_{n}$.
This completes the iteration.  
The user has complete flexibility in the choice of $\sigma_{n}$.  One possibility is $\sigma_{n} =\infty$.  Another,
which might be called a \emph{Rayleigh-quotient pole}, is $\sigma_{n} = a_{11}/b_{11}$.     

The cost of one iteration of the basic algorithm is $n$ moves or $O(n^{2})$ flops.
With any of the standard shifting strategies, e.g.\ Rayleigh-quotient shift, repeated iterations will normally cause
rapid convergence of an eigenvalue at the bottom of the pencil.  Typically 
$a_{n,n-1} \to 0$ and  $b_{n,n-1} \to 0$ quadratically, leaving $a_{nn}/b_{nn}$ as an eigenvalue and allowing deflation
of the problem.  After $n-1$ deflations, all of the eigenvalues will have been found.  

There are numerous variations on the basic algorithm.  For example, it can be turned
upside down.  We can pick a shift, say $\rho = a_{11}/b_{11}$, insert it at the bottom of the pencil, 
and chase it to the top.  
Since we can do this, then why not chase shifts in both directions at once?   Some possibilities along these lines will
be discussed in Section~\ref{sec:variations}.

\subsection*{Relationship to the QZ algorithm}

We now show that when the basic algorithm is applied to a pair that has all poles infinity, it reduces to the 
single-shift version of the Moler/Stewart QZ algorithm.  Consider a Hessenberg-triangular pair
\begin{displaymath}
\parbox{2.cm}{
\begin{tikzpicture}[scale=1.66,y=-1cm]
\draw (-.15,-.1) -- (-.2,-.1) -- (-.2,0.7) -- (-.15,0.7);
\draw (.75,-.1) -- (.8,-.1) -- (.8,0.7) -- (.75,0.7);
\foreach \j in {0,...,3}{
   \foreach \i in {\j,...,3}{\node at (\i/5,\j/5)
     [align=center,scale=1.0]{$\times$};}}
\foreach \j in {0,...,2}{\node at (\j/5,\j/5+.2)
     [align=center,scale=1.0]{$\times$};}
\phantom{\color{red}\node at (0,.4){$+$};\color{black}}
\phantom{\color{red}\node at (.2,.6){$+$};\color{black}}
\end{tikzpicture} 
} \qquad \qquad
\parbox{2.cm}{
\begin{tikzpicture}[scale=1.66,y=-1cm]
\draw (-.15,-.1) -- (-.2,-.1) -- (-.2,0.7) -- (-.15,0.7);
\draw (.75,-.1) -- (.8,-.1) -- (.8,0.7) -- (.75,0.7);
\foreach \j in {0,...,3}{
   \foreach \i in {\j,...,3}{\node at (\i/5,\j/5)
     [align=center,scale=1.0]{$\times$};}}
\phantom{\color{red}\node at (0,.4){$+$};\color{black}}
\phantom{\color{red}\node at (.2,.6){$+$};\color{black}}
\end{tikzpicture} 
} 
\end{displaymath} 
which has poles $\infty$, $\infty$, and $\infty$.  An iteration of the basic algorithm 
begins by choosing a shift $\rho$ 
and inserting it into the pair at the top by a move of type I.   The transformation is $A \to Q_{1}^{*}A$, 
$B \to Q_{1}^{*}B$, where $Q_{1}$ satisfies (\ref{eq:qzstart}).  This is exactly the same as the transformation
that starts single-shift QZ  \cite[p.~537]{Wat10}.  It alters the first two rows of the matrices, 
so the transformed matrices have the form
\begin{equation}\label{eq:b21bulge}
\parbox{2.cm}{
\begin{tikzpicture}[scale=1.66,y=-1cm]
\draw (-.15,-.1) -- (-.2,-.1) -- (-.2,0.7) -- (-.15,0.7);
\draw (.75,-.1) -- (.8,-.1) -- (.8,0.7) -- (.75,0.7);
\foreach \j in {0,...,3}{
   \foreach \i in {\j,...,3}{\node at (\i/5,\j/5)
     [align=center,scale=1.0]{$\times$};}}
\foreach \j in {0,...,2}{\node at (\j/5,\j/5+.2)
     [align=center,scale=1.0]{$\times$};}
\phantom{\color{red}\node at (0,.4){$+$};\color{black}}
\phantom{\color{red}\node at (.2,.6){$+$};\color{black}}
\end{tikzpicture} 
} \qquad \qquad
\parbox{2.cm}{
\begin{tikzpicture}[scale=1.66,y=-1cm]
\draw (-.15,-.1) -- (-.2,-.1) -- (-.2,0.7) -- (-.15,0.7);
\draw (.75,-.1) -- (.8,-.1) -- (.8,0.7) -- (.75,0.7);
\foreach \j in {0,...,3}{
   \foreach \i in {\j,...,3}{\node at (\i/5,\j/5)
     [align=center,scale=1.0]{$\times$};}}
\color{red}\node at (0,.2){$+$};\color{black}
\phantom{\color{red}\node at (.2,.6){$+$};\color{black}}
\end{tikzpicture} 
}. 
\end{equation}  
The triangular form of $B$ has been disturbed, but this is still a Hessenberg pair.  Its poles are 
$\rho$, $\infty$, $\infty$.   (We will continue to refer to the matrices as ``$A$" and ``$B$'', 
even though they change in the course of the iteration.)
The next step of the basic algorithm is a move of type II that interchanges the 
pole $\rho$ with the adjacent pole $\infty$, resulting in 
\begin{equation}\label{eq:b32bulge}
\parbox{2.cm}{
\begin{tikzpicture}[scale=1.66,y=-1cm]
\draw (-.15,-.1) -- (-.2,-.1) -- (-.2,0.7) -- (-.15,0.7);
\draw (.75,-.1) -- (.8,-.1) -- (.8,0.7) -- (.75,0.7);
\foreach \j in {0,...,3}{
   \foreach \i in {\j,...,3}{\node at (\i/5,\j/5)
     [align=center,scale=1.0]{$\times$};}}
\foreach \j in {0,...,2}{\node at (\j/5,\j/5+.2)
     [align=center,scale=1.0]{$\times$};}
\phantom{\color{red}\node at (0,.4){$+$};\color{black}}
\phantom{\color{red}\node at (.2,.6){$+$};\color{black}}
\end{tikzpicture} 
} \qquad\qquad 
\parbox{2.cm}{
\begin{tikzpicture}[scale=1.66,y=-1cm]
\draw (-.15,-.1) -- (-.2,-.1) -- (-.2,0.7) -- (-.15,0.7);
\draw (.75,-.1) -- (.8,-.1) -- (.8,0.7) -- (.75,0.7);
\foreach \j in {0,...,3}{
   \foreach \i in {\j,...,3}{\node at (\i/5,\j/5)
     [align=center,scale=1.0]{$\times$};}}
\phantom{\color{red}\node at (0,.2){$+$};\color{black}}
\color{red}\node at (.2,.4){$+$};\color{black}
\end{tikzpicture} 
}, 
\end{equation}  
a Hessenberg pair with poles $\infty$, $\rho$, $\infty$.   The transformation has the form 
$A \to Q_{2}^{*}AZ_{1}$, $B \to Q_{2}^{*}BZ_{1}$, with appropriately 
chosen core transformations $Z_{1}$ and $Q_{2}$.  Let us consider now how things look if we apply the cores
one at a time.   Starting from the configuration shown in (\ref{eq:b21bulge}), first apply $Z_{1}$
on the right.   This acts on columns one and two of each matrix and produces
\begin{displaymath}
\parbox{2.cm}{
\begin{tikzpicture}[scale=1.66,y=-1cm]
\draw (-.15,-.1) -- (-.2,-.1) -- (-.2,0.7) -- (-.15,0.7);
\draw (.75,-.1) -- (.8,-.1) -- (.8,0.7) -- (.75,0.7);
\foreach \j in {0,...,3}{
   \foreach \i in {\j,...,3}{\node at (\i/5,\j/5)
     [align=center,scale=1.0]{$\times$};}}
\foreach \j in {0,...,2}{\node at (\j/5,\j/5+.2)
     [align=center,scale=1.0]{$\times$};}
\color{red}\node at (0,.4){$+$};\color{black}
\phantom{\color{red}\node at (.2,.6){$+$};\color{black}}
\end{tikzpicture} 
} \qquad\qquad 
\parbox{2.cm}{
\begin{tikzpicture}[scale=1.66,y=-1cm]
\draw (-.15,-.1) -- (-.2,-.1) -- (-.2,0.7) -- (-.15,0.7);
\draw (.75,-.1) -- (.8,-.1) -- (.8,0.7) -- (.75,0.7);
\foreach \j in {0,...,3}{
   \foreach \i in {\j,...,3}{\node at (\i/5,\j/5)
     [align=center,scale=1.0]{$\times$};}}
\phantom{\color{red}\node at (0,.2){$+$};\color{black}}
\phantom{\color{red}\node at (.2,.6){$+$};\color{black}}
\end{tikzpicture} 
}. 
\end{displaymath}  
The entry $b_{21}$ must now be zero.  This is so because, as we know, after the application of $Q_{2}^{*}$ on 
the left, $b_{21}$ must be zero, as shown in (\ref{eq:b32bulge}).  The left multiplication by $Q_{2}^{*}$ cannot do this job, 
so it must have been done by $Z_{1}$.  At the same time, $Z_{1}$ must produce a bulge
at $a_{31}$.  This proves that $Z_{1}$ is exactly the same transformation as is used at this point
in the QZ bulge chase.  

Now, when we apply $Q_{2}^{*}$ on the left, it operates on rows two and three.  
It must set $a_{31}$ to zero and create a new bulge at $b_{32}$ to arrive at (\ref{eq:b32bulge}).
Thus $Q_{2}$ is exactly the same transformation as is used at this point in the QZ bulge chase.

The next step is a move of type II that transforms (\ref{eq:b32bulge}) to 
\begin{equation}\label{eq:b43bulge}
\parbox{2.cm}{
\begin{tikzpicture}[scale=1.66,y=-1cm]
\draw (-.15,-.1) -- (-.2,-.1) -- (-.2,0.7) -- (-.15,0.7);
\draw (.75,-.1) -- (.8,-.1) -- (.8,0.7) -- (.75,0.7);
\foreach \j in {0,...,3}{
   \foreach \i in {\j,...,3}{\node at (\i/5,\j/5)
     [align=center,scale=1.0]{$\times$};}}
\foreach \j in {0,...,2}{\node at (\j/5,\j/5+.2)
     [align=center,scale=1.0]{$\times$};}
\phantom{\color{red}\node at (0,.4){$+$};\color{black}}
\phantom{\color{red}\node at (.2,.6){$+$};\color{black}}
\end{tikzpicture} 
} \qquad\qquad 
\parbox{2.cm}{
\begin{tikzpicture}[scale=1.66,y=-1cm]
\draw (-.15,-.1) -- (-.2,-.1) -- (-.2,0.7) -- (-.15,0.7);
\draw (.75,-.1) -- (.8,-.1) -- (.8,0.7) -- (.75,0.7);
\foreach \j in {0,...,3}{
   \foreach \i in {\j,...,3}{\node at (\i/5,\j/5)
     [align=center,scale=1.0]{$\times$};}}
\phantom{\color{red}\node at (0,.2){$+$};\color{black}}
\phantom{\color{red}\node at (.2,.4){$+$};\color{black}}
\color{red}\node at (.4,.6){$+$};\color{black}
\end{tikzpicture} 
}, 
\end{equation}  
a Hessenberg pair with poles $\infty$, $\infty$, $\rho$.  The transformation has the form 
$A \to Q_{3}^{*}AZ_{2}$,  $B \to Q_{3}^{*}BZ_{2}$.  
Again we could look at what happens if we apply the cores one at a time, first $Z_{2}$, then $Q_{3}^{*}$, 
and we would find as before that these are exactly the same transformations as in a $QZ$ bulge chase.  

In our little example, we have now reached the bottom.  In a larger example, we would continue moves of type
II, pushing the pole $\rho$ downward, and at each step we would have the same situation.  The final step is 
a move of type I that removes $\rho$ from the bottom of the pencil, replacing it by a pole $\infty$.  This is exactly
the transform, acting on columns $n-1$ and $n$, that sets  $b_{n,n-1}$  (the entry $b_{43}$ entry in (\ref{eq:b43bulge}))
to zero.  Again this is exactly the same as the transformation that completes the QZ bulge chase.  The pair is now
in Hessenberg-triangular form.  

We have demonstrated that the basic algorithm reduces to the single-shift QZ algorithm in the case when 
all of the poles are infinite.  

\section{Convergence theory}

In the convergence theorems in this paper we make the blanket (and generically valid) assumption that none of
the poles or shifts that are mentioned are eigenvalues of the pencil.   We often find it convenient to 
assume that $B$ is nonsingular. 

The mechanism that drives all variants of Francis's algorithm is nested subspace 
iteration with changes of coordinate system 
\cite[p.~431]{Wat10}, \cite[p.~399]{Wat11},  \cite[Thm~2.2.3]{AuMaRoVaWa18}.  
As a specific example, let us consider a single step of the QZ algorithm
with shift $\rho$ applied to a Hessenberg-triangular pencil $A - \lambda B$, 
yielding a new pencil $\hat{A} - \lambda \hat{B}$ with 
\begin{equation}\label{eq:qzcoord}
\hat{A} - \lambda \hat{B} = Q^{*}(A - \lambda B)Z.
\end{equation}
First we define some nested sequences of subspaces.  
For $k=1$, \ldots, $n$, define 
\begin{displaymath}
\eee_{k} = \spn{e_{1}, \ldots, e_{k}},
\end{displaymath}
where $e_{1}$, \ldots, $e_{n}$ are the standard basis vectors.  Then define 
\begin{displaymath}
\cue_{k} = Q\eee_{k} \quad\mbox{and}\quad \zee_{k} = Z\eee_{k}.
\end{displaymath}
Thus $\cue_{k}$ (resp.\ $\zee_{k}$) is the space spanned by the first $k$ columns of $Q$ (resp.\ $Z$).

\begin{theorem}\label{thm:qzspace}
A single step of the $QZ$ algorithm with shift $\rho$ effects nested subspace iterations  
\begin{displaymath} 
\cue_{k} = (AB^{-1} - \rho I)\eee_{k}, \qquad \zee_{k} = (B^{-1}A - \rho I)\eee_{k}, \qquad k = 1,\ldots,n-1.
\end{displaymath}
The change of coordinate system (\ref{eq:qzcoord}) transforms both $\cue_{k}$ and $\zee_{k}$ back to $\eee_{k}$.
\end{theorem}

We call this a \emph{convergence theorem} even though it makes no mention of convergence.   Theorems
like this can be used together with the convergence theory of subspace iteration to draw conclusions about
the convergence of the algorithm, as explained in \cite{Wat07,Wat10,Wat11} and elsewhere.  

Camps, Meerbergen, and Vandebril \cite[Thm.~6.1]{CaMeVa19a} proved a result like Theorem~\ref{thm:qzspace} 
for the basic algorithm.  
The scenario is similar.  The iteration begins with a 
proper Hessenberg pair $(A,B)$ with poles $\sigma_{1}$, \ldots, $\sigma_{n-1}$, employs a shift $\rho$, and ends
with a new proper Hessenberg pair $(\hat{A},\hat{B})$ with poles $\sigma_{2}$, \ldots, $\sigma_{n}$.  The old and new
pairs are related by a unitary equivalence transformation of the form (\ref{eq:qzcoord}).

\begin{theorem}\label{thm:basicalgspace}
A single step of the basic algorithm with shift $\rho$, starting with a proper Hessenberg pair $(A,B)$ with poles
$\sigma_{1}$, \ldots, $\sigma_{n-1}$ and ending with $(\hat{A},\hat{B})$ with poles $\sigma_{2}$, \ldots, $\sigma_{n}$ 
effects nested subspace iterations 
\begin{displaymath}
\cue_{k} = (A - \rho B)(A - \sigma_{k}B)^{-1}\eee_{k}, \quad \zee_{k} = (A - \sigma_{k+1}B)^{-1}(A - \rho B)\eee_{k}, 
\quad k = 1, \ldots, n-1.
\end{displaymath}
The change of coordinate system (\ref{eq:qzcoord}) transforms both $\cue_{k}$ and $\zee_{k}$ back to $\eee_{k}$.\end{theorem}

This theorem was proved in \cite{CaMeVa19a}, but we will also provide a proof based on our new theory in
Section~\ref{sec:moveuse}.
Comparing this with Theorem~\ref{thm:qzspace}, we see that the inclusion of poles gives extra freedom that might
be used to improve convergence.   

Now consider Theorem~\ref{thm:basicalgspace}  in the case when all of the poles are infinite.  
When $\sigma_{k}=\infty$, 
the operator $(A - \rho B)(A - \sigma_{k}B)^{-1}$ becomes (when appropriately rescaled)  $(A - \rho B)B^{-1} = 
AB^{-1} - \rho I$.  Similarly $(A - \sigma_{k+1}B)^{-1}(A - \rho B)$ becomes $B^{-1}(A - \rho B) = B^{-1}A - \rho I$.  
These operators are exactly the ones that appear in Theorem~\ref{thm:qzspace},  just as we would expect.  

Although the QZ algorithm is a special case of the basic algorithm, 
there is an important difference in their implementation.  
The QZ algorithm acts on proper Hessenberg-triangular pencils.  It is a bulge-chasing algorithm.  The initial
equivalence transformation of each iteration creates a bulge in the Hessenberg-triangular form.  The rest of the 
iteration consists of equivalence transformations that chase the bulge back and forth between $A$ and $B$ until
it finally disappears off of the bottom of the pencil.  At that point the Hessenberg-triangular form has been restored
and the iteration is complete.   The QZ algorithm can also be implemented as a \emph{core-chasing} algorithm,
as is shown in \cite{AuMaRoVaWa18} and \cite{AuMaRoVaWa19}, but the situation is the same:  The 
Hessenberg-triangular form is disturbed at the beginning of the iteration and not restored until the very end.

Now let us contrast this with what happens in the basic algorithm (with infinite poles or otherwise).  
The basic algorithm 
operates on proper Hessenberg pairs, in which neither matrix is required to be triangular.  Each 
iteration starts with a move of type I, performs a sequence of moves of type II, and ends with a move of type I.  
These moves do not disturb the Hessenberg form; it is preserved throughout.  
This implies that we can think of each move as a ``mini iteration'' and ask whether we can obtain a result like
Theorem~\ref{thm:qzspace} or \ref{thm:basicalgspace} for each individual move of type I or II.   It turns out that 
we can.

Each move of either type is an equivalence transform of the form  
\begin{displaymath}
\hat{A} = Q_{j}^{*}A Z_{j-1} \quad \hat{B} = Q_{j}^{*}BZ_{j-1}.
\end{displaymath}
The case $j=1$ denotes a move of type I, and we have $Z_{0} = I$.  
The case $j=n$ also denotes a type I move, and in this case $Q_{n} = I$.
The cases $j=2$, \ldots, $n-1$ are of type II.   
Suppose $(A,B)$ has poles $\sigma_{1}$, \ldots, $\sigma_{n-1}$.  A
move of type II interchanges poles $\sigma_{j-1}$ and $\sigma_{j}$.  
For the moves of type I, in the case $j=1$, suppose the pole $\sigma_{1}$
is replaced by a new pole $\sigma_{0}$; in the case $j=n$, suppose $\sigma_{n-1}$
is replaced by a new pole $\sigma_{n}$.  With this notation we can cover both 
types of move by a single theorem. 

As above we define sequences of nested subspaces $(\cue_{k})$ and $(\zee_{k})$, where $\cue_{k}$  
(resp.\ $\zee_{k}$) is the space spanned by the first $k$ columns of $Q_{j}$ (resp.\ $Z_{j-1}$).  But note that,
because $Q_{j}$ and $Z_{j-1}$ are core transformations,  
these spaces are mostly trivial in this setting:  $\cue_{k} = \eee_{k}$ except when $k=j$, and $\zee_{k} = \eee_{k}$
except when $k=j-1$.

\begin{theorem}\label{thm:movespace}
Using notation and terminology established directly above, the move 
\begin{equation}\label{eq:movetrans}
\hat{A} - \lambda\hat{B} = Q_{j}^{*}(A - \lambda B) Z_{j-1} 
\end{equation}
effects nested subspace iterations that are, however, mostly trivial.  The nontrivial 
actions are 
\begin{displaymath}
\cue_{j} = (A - \sigma_{j-1} B)(A - \sigma_{j} B)^{-1}\eee_{j}
\end{displaymath}
and 
\begin{displaymath}
\zee_{j-1} = (A - \sigma_{j} B)^{-1}(A - \sigma_{j-1}B)\eee_{j-1}.
\end{displaymath}
The change of coordinate system (\ref{eq:movetrans}) transforms  $\cue_{j}$ back to $\eee_{j}$ and 
$\zee_{j-1}$ back to $\eee_{j-1}$.
\end{theorem}

The proof of Theorem~\ref{thm:movespace} makes use of rational Krylov subspaces.
Given $C\in\cnxn$ and $v\in \cn$, the standard \emph{Krylov subspaces} $\kay_{j}(C,v)$ are 
defined by 
\begin{displaymath}
\kay_{j}(C,v) = \spn{v,Cv,C^{2}v, \ldots, C^{j-1}v}, \qquad j = 1,\, 2,\,\ldots, \, n .
\end{displaymath}
Given an ordered set of poles $[\sigma_{1},\, \sigma_{2},\, \ldots,\, \sigma_{n-1}]$, none in the spectrum 
of $C$, the 
\emph{rational Krylov subspaces}  $\kay_{j}(C,v,[\sigma_{1}, \ldots, \sigma_{j-1}])$ are defined by 
\begin{eqnarray*}
\kay_{1}(C,v,[\,]) & = & \spn{v}, \\
\kay_{2}(C,v, [\sigma_{1}]) & = & \spn{v,(C - \sigma_{1}I)^{-1}v}, \\
\kay_{3}(C,v, [\sigma_{1}, \sigma_{2}]) & = & \spn{v,(C - \sigma_{1}I)^{-1}v,(C-\sigma_{2}I)^{-1}(C - \sigma_{1}I)^{-1}v},
\end{eqnarray*}
and in general 
\begin{displaymath}
\kay_{j}(C,v, [\sigma_{1}, \ldots, \sigma_{j-1}]) = 
\spn{v,(C - \sigma_{1}I)^{-1}v, \ldots, \left(\prod_{i=1}^{j-1}(C-\sigma_{i}I)^{-1}\right)v}.
\end{displaymath}
Making the abbreviation $C(\sigma) = C - \sigma I$, we can rewrite this as 
\begin{displaymath}
\kay_{j}(C,v, [\sigma_{1}, \ldots, \sigma_{j-1}]) = 
\prod_{i=1}^{j-1}C(\sigma_{i})^{-1}\,\spn{\prod_{i=1}^{j-1}C(\sigma_{i})v,\prod_{i=2}^{j-1}C(\sigma_{i})v, \ldots, v}.
\end{displaymath}
The span on the right-hand side involves only positive powers of $C$, so the shifts are irrelevant; it is just the standard
Krylov subspace $\kay_{j}(C,v)$.  Therefore 
\begin{equation}\label{def:altratkry}
\kay_{j}(C,v, [\sigma_{1}, \ldots, \sigma_{j-1}]) =\left(\prod_{i=1}^{j-1} (C - \sigma_{i}I)^{-1}\right)\kay_{j}(C,v).
\end{equation}

Given a pair $(A,B)$ with $B$ nonsingular, we define \emph{rational Krylov subspaces}  
\begin{displaymath}
\kay_{j}(A,B,v,[\sigma_{1},\ldots,\sigma_{j-1}]) \quad\mbox{and}\quad 
\ehl_{j}(A,B,v,[\sigma_{1},\ldots,\sigma_{j-1}])
\end{displaymath}
associated with the pair by 
\begin{displaymath}
\kay_{j}(A,B,v,[\sigma_{1},\ldots,\sigma_{j-1}]) = \kay_{j}(AB^{-1},v,[\sigma_{1},\ldots,\sigma_{j-1}])
\end{displaymath}
and
\begin{displaymath}
\ehl_{j}(A,B,v,[\sigma_{1},\ldots,\sigma_{j-1}]) = 
\kay_{j}(B^{-1}A,v,[\sigma_{1},\ldots,\sigma_{j-1}]),
\end{displaymath}
$j=1$, \ldots, $n$.  We have assumed for convenience that $B$ is nonsingular.  See \cite{CaMeVa19a} for a 
definition of these spaces that does not require this assumption.   We are using the symbol $\kay_{j}$ to denote
several different types of Krylov subspaces.  The meaning in each case is uniquely determined by the number
and type of arguments.

We will make use of the following result, which is Theorem~5.6 
in \cite{CaMeVa19a}.

\begin{proposition}\label{prop:ejkrylov}
Let $(A,B)$ be a proper upper Hessenberg pair with poles $[\sigma_{1},\ldots,\sigma_{n-1}]$.  Let 
$\eee_{j}  = \spn{e_{1},\ldots,e_{j}}$, as before.  Then for $j=1$, \ldots, $n-1$, 
\begin{displaymath}
\eee_{j} = \kay_{j}(A,B,e_{1}, [\sigma_{1},\ldots,\sigma_{j-1}]) = \ehl_{j}(A,B,e_{1}, [\sigma_{2}, \ldots,\sigma_{j}]).
\end{displaymath}
\end{proposition}
See \cite{CaMeVa19a} for the proof.  Notice that in the $\ehl_{j}$ spaces the poles are 
$[\sigma_{2},\ldots,\sigma_{j}]$, starting from $\sigma_{2}$.
With Proposition~\ref{prop:ejkrylov} in hand, we can prove Theorem~\ref{thm:movespace}.  

\begin{proof}
Proposition~\ref{prop:firstcol} shows that $Q_{1}e_{1} =  \delta\,(A - \sigma_{0} B)(A - \sigma_{1}B)^{-1}e_{1}$
for some nonzero $\delta$.  This establishes the case $j=1$ of Theorem~\ref{thm:movespace}.  

Now consider $j>1$.  The transformation $\hat{A} - \lambda \hat{B} = Q_{j}^{*}(A - \lambda B)Z_{j-1}$ interchanges 
poles $\sigma_{j-1}$ and $\sigma_{j}$, so the ordered pole set of $(\hat{A},\hat{B})$ is 
$$[\sigma_{1},\ldots,\sigma_{j-2},\sigma_{j},\sigma_{j-1},\sigma_{j+1}, \ldots,\sigma_{n-1}].$$  
Applying 
Proposition~\ref{prop:ejkrylov} to $(\hat{A},\hat{B})$ we have 
\begin{displaymath}
\eee_{j} = \kay_{j}(\hat{A},\hat{B},e_{1},[\sigma_{1},\ldots,\sigma_{j-2},\sigma_{j}]).
\end{displaymath}
Therefore 
\begin{eqnarray*}
\cue_{j} = Q_{j}\eee_{j} & = & Q_{j}\kay_{j}(\hat{A},\hat{B},e_{1},[\sigma_{1}, \ldots,\sigma_{j-2},\sigma_{j}])  \\
& = & \kay_{j}(A,B,Q_{j}e_{1},[\sigma_{1}, \ldots,\sigma_{j-2},\sigma_{j}]),
\end{eqnarray*}
Noting that $Q_{j}e_{1} = e_{1}$, using the abbreviations $C = AB^{-1}$ and $C(\sigma) = AB^{-1} - \sigma I$, 
and using  (\ref{def:altratkry}) twice, we obtain 
\begin{eqnarray*}
\cue_{j} & = & \kay_{j}(AB^{-1},e_{1},[\sigma_{1}, \ldots,\sigma_{j-2},\sigma_{j}]) \\
& = &  C(\sigma_{j})^{-1} \left( \prod_{i=1}^{j-2}C(\sigma_{i})^{-1}\right)\kay_{j}(C,e_{1}) \\
& = &  C(\sigma_{j})^{-1}C(\sigma_{j-1})\left(\prod_{i=1}^{j-1}C(\sigma_{i})^{-1}\right)\kay_{j}(C,e_{1}) \\
& = &  C(\sigma_{j})^{-1}C(\sigma_{j-1})\kay_{j}(A,B,e_{1},[\sigma_{1},\ldots,\sigma_{j-1}]) \\
& = &  C(\sigma_{j})^{-1}C(\sigma_{j-1})\eee_{j}. 
\end{eqnarray*}
In the final step we used Proposition~\ref{prop:ejkrylov} again.  Since
$C(\sigma_{j})^{-1}C(\sigma_{j-1}) = C(\sigma_{j-1})C(\sigma_{j})^{-1} 
= (AB^{-1} - \sigma_{j-1}I)(AB^{-1} - \sigma_{j}I)^{-1} = (A - \sigma_{j-1}B)(A - \sigma_{j}B)^{-1}$, 
we get the desired result 
$\cue_{j} = (A - \sigma_{j-1}B)(A - \sigma_{j}B)^{-1}\eee_{j}$.

In this argument we have assumed that $B^{-1}$ exists.  However, the result also holds for singular $B$
by a continuity argument.  

Now consider the spaces $\zee_{j-1}$.   In the case $j=2$ we have $\hat{A}  - \lambda \hat {B}
= Q_{2}^{*}(A - \lambda B)Z_{1}$.   Substituting $\lambda = \sigma_{2}$ and solving for $Z_{1}$,
we have $Z_{1} = (A - \sigma_{2}B)^{-1}Q_{2}(\hat{A} - \sigma_{2}\hat{B})$.  The ordered pole set for 
$(\hat{A},\hat{B})$ is $[\sigma_{2},\sigma_{1},\sigma_{3}, \ldots,\sigma_{n-1}]$, so 
$(\hat{A} - \sigma_{2}\hat{B})e_{1} = \gamma e_{1}$ for some nonzero $\gamma$.  Similarly  
$(A - \sigma_{1}B)e_{1} = \delta e_{1}$ for some nonzero $\delta$.  Therefore  
$$Z_{1} e_{1} = \gamma (A - \sigma_{2}B)^{-1}Q_{2}e_{1} = \gamma (A - \sigma_{2}B)^{-1}e_{1} = 
\gamma\delta^{-1}(A - \sigma_{2}B)^{-1}(A - \sigma_{1}B)e_{1}.$$  This proves that 
\begin{displaymath}
\zee_{1} = (A - \sigma_{2}B)^{-1}(A - \sigma_{1}B)\eee_{1},
\end{displaymath}
as desired.

For $j > 2$ we have $\hat{A} - \lambda \hat{B} = Q_{j}^{*}(A - \lambda B)Z_{j-1}$.
Arguing just as we did for  $\cue_{j}$, we have
\begin{eqnarray*}
\zee_{j-1} = Z_{j-1}\eee_{j-1} & = & Z_{j-1}\ehl_{j-1}(\hat{A},\hat{B},e_{1},[\sigma_{2}, \ldots, \sigma_{j-2}, \sigma_{j}]) \\
& = & \ehl_{j-1}(A,B,Z_{j-1}e_{1},[\sigma_{2}, \ldots, \sigma_{j-2},\sigma_{j}]).
\end{eqnarray*}
Using  $Z_{j-1}e_{1} = e_{1}$, and making the abbreviations $D = B^{-1}A$ and 
$D(\sigma) = B^{-1}A - \sigma I$, we have 
\begin{eqnarray*}
\zee_{j-1} & = & \ehl_{j-1}(A,B,e_{1},[\sigma_{2}, \ldots, \sigma_{j-2},\sigma_{j}]) \\
& = & \kay_{j-1}(D,e_{1},[\sigma_{2}, \ldots, \sigma_{j-2},\sigma_{j}]) \\
& = & D(\sigma_{j})^{-1}\left( \prod_{i=2}^{j-2} D(\sigma_{i})^{-1}\right)\kay_{j-1}(D,e_{1}) \\
& = & D(\sigma_{j})^{-1} D(\sigma_{j-1})\left( \prod_{i=2}^{j-1} D(\sigma_{i})^{-1}\right)\kay_{j-1}(D,e_{1}) \\
& = & D(\sigma_{j})^{-1} D(\sigma_{j-1}) \kay_{j-1}(D,e_{1},[\sigma_{2},\ldots,\sigma_{j-1}]) \\
& = & (A - \sigma_{j}B)^{-1}(A - \sigma_{j-1}B)\ehl_{j-1}(A,B,e_{1}, [\sigma_{2}, \ldots, \sigma_{j-1}]) \\
& = & (A - \sigma_{j}B)^{-1}(A - \sigma_{j-1}B)\eee_{j-1}. 
\end{eqnarray*}

\hfill\end{proof}

\begin{remark}
We used Proposition~\ref{prop:firstcol} to prove the case $j=1$, but we did not use 
Proposition~\ref{prop:lastrow}.  In connection with this we remark that Theorem~\ref{thm:movespace} 
immediately implies the dual results
\begin{displaymath}
\cue_{j}^{\perp} = (A^{*} - \overline{\sigma}_{j-1} B^{*})^{-1}(A^{*} - \overline{\sigma}_{j} B^{*})\eee_{j}^{\perp}
\end{displaymath}
and 
\begin{displaymath}
\zee_{j-1}^{\perp} = (A^{*} - \overline{\sigma}_{j} B^{*})(A^{*} - \overline{\sigma}_{j-1}B^{*})^{-1}\eee_{j-1}^{\perp},
\end{displaymath}
obtained by noting that $\you = C\ess$ if and only if $\you^{\perp} = (C^{*})^{-1}\ess^{\perp}$.  We could equally well have 
derived the dual results first and then deduced Theorem~\ref{thm:movespace}.  In that case we would use 
Proposition~\ref{prop:lastrow} to prove the case $j=n$, and not use Proposition~\ref{prop:firstcol} at all.  From 
Proposition~\ref{prop:lastrow}  with $\tau = \sigma_{n}$ we have immediately
\begin{displaymath}
Z_{n-1}e_{n} = \overline{\delta}\,(A^{*} - \overline{\sigma}_{n}B^{*})(A^{*} - \overline{\sigma}_{n-1}B^{*})^{-1}e_{n}, 
\end{displaymath}  
which implies
\begin{displaymath}
\zee_{n-1}^{\perp} = (A^{*} - \overline{\sigma}_{n}B^{*})(A^{*} - 
\overline{\sigma}_{n-1}B^{*})^{-1}\eee_{n-1}^{\perp}, 
\end{displaymath}
the case $j = n$ of the dual result. 
\end{remark}

\section{Using Theorem~\ref{thm:movespace}}\label{sec:moveuse}

In all of the convergence theorems of the previous section we have actions of the form $\cue_{k} = r(AB^{-1})\eee_{k}$
and $\zee_{k} = r(B^{-1}A)\eee_{k}$, where $r$ is a rational function, e.g.\ $r(z) = (z - \sigma_{j-1})/(z - \sigma_{j})$.   
In the following lemma the functions $r$ and $s$ can be any functions  
defined on the spectrum of the pencil $A - \lambda B$, but in our applications they will always be rational.  
In this case, being defined on the
spectrum of $A - \lambda B$ just means that none of the poles are eigenvalues.  

\begin{lemma}\label{lem:accumstep}
Consider two successive changes of coordinate system 
\begin{displaymath}
\tilde{A} - \lambda\tilde{B} = \tilde{Q}^{*}(A - \lambda B)\tilde{Z} \quad\mbox{and}\quad
\hat{A} - \lambda \hat{B} = \hat{Q}^{*}(\tilde{A} - \lambda\tilde{B})\hat{Z},
\end{displaymath}
so that
\begin{displaymath}
\hat{A} - \lambda \hat{B} = Q^{*}(A - \lambda B)Z,
\quad \mbox{where}\quad  Q = \tilde{Q}\hat{Q} \quad\mbox{and}\quad Z = \tilde{Z}\hat{Z}.
\end{displaymath}
For $k=1$, \ldots, $n-1$, if 
\begin{displaymath}
\tilde{Q}\eee_{k} = r(AB^{-1})\eee_{k} \quad\mbox{and}\quad \hat{Q}\eee_{k} = s(\tilde{A}\tilde{B}^{-1})\eee_{k},
\end{displaymath} 
then
\begin{displaymath}
Q\eee_{k} = sr(AB^{-1})\eee_{k},
\end{displaymath}
 where $sr$ is the pointwise product of $s$ and $r$.
If 
\begin{displaymath}
\tilde{Z}\eee_{k} = r(B^{-1}A)\eee_{k} \quad\mbox{and} \quad \hat{Z}\eee_{k} = s(\tilde{B}^{-1}\tilde{A})\eee_{k}, 
\end{displaymath}
then 
\begin{displaymath}
Z\eee_{k} = sr(B^{-1}A)\eee_{k}.
\end{displaymath}  
\end{lemma}

\begin{proof}
Noting that $\tilde{Q}\,s(\tilde{A}\tilde{B}^{-1}) = s(AB^{-1})\tilde{Q}$, we have 
$$Q\eee_{k} = \tilde{Q}\hat{Q}\eee_{k} = \tilde{Q}\,s(\tilde{A}\tilde{B}^{-1})\eee_{k} = s(AB^{-1})\tilde{Q}\eee_{k} = s(AB^{-1})r(AB^{-1})\eee_{k},$$ 
so 
$Q\eee_{k} = sr(AB^{-1})\eee_{k}$.   
The result for $Z\eee_{k}$ is proved similarly, using 
$\tilde{Z}\,s(\tilde{B}^{-1}\tilde{A}) = s(B^{-1}A)\tilde{Z}$.
\hfill\end{proof}

Clearly this lemma can be extended by induction to three or more successive changes of coordinate system, 
and that's how we are going to use it.  

\subsection*{Proof of Theorem~\ref{thm:basicalgspace}}

As a first application of Theorem~\ref{thm:movespace}, we show that it can be used to prove 
Theorem~\ref{thm:basicalgspace}.  

According to Theorem~\ref{thm:basicalgspace}, 
for each $k$ the basic algorithm effects 
a transformation 
\begin{equation}\label{eq:kstep}
\cue_{k} = (A - \rho B)(A - \sigma_{k}B)^{-1}\eee_{k}.
\end{equation}
Let us see why this is so.  Recall that the basic algorithm begins with a move of type I that introduces 
the shift $\rho$ as a pole at the top of the pencil.  It then does a sequence of moves of type II that  swap 
$\rho$ with the other poles one by one.  For a given $k$, 
most of these moves have no effect on $\eee_{k}$.  
The only exception is the $k$th move, the case $j=k$ in Theorem~\ref{thm:movespace}.  
This is where we need to focus.  

One iteration of the basic algorithm performs the equivalence 
\begin{displaymath}
\hat{A} - \lambda \hat{B} = Q^{*}(A - \lambda B)Z,
\end{displaymath}
where $Q$ and $Z$ are products of core transformations:
\begin{displaymath}
Q = Q_{1}Q_{2}\cdots Q_{n-1}, \qquad Z = Z_{1}Z_{2} \cdots Z_{n-1}.
\end{displaymath}
The core $Q_{1}$ is the one that replaces pole $\sigma_{1}$ with the shift $\rho$.  $Q_{2}$ (together with $Z_{1}$)
swaps $\rho$ with $\sigma_{2}$, $Q_{3}$ (together with $Z_{2}$) swaps $\rho$ with $\sigma_{3}$, and so on.
$Z_{n-1}$ removes $\rho$ and installs a new pole $\sigma_{n}$.  We are interested in the action of $Q_{k}$
(together with $Z_{k-1}$), which swaps $\rho$ with $\sigma_{k}$.  Thus we factor $Q$ and $Z$ as 
\begin{displaymath}
Q = \tilde{Q}Q_{k}\hat{Q}, \qquad  Z = \tilde{Z}Z_{k-1}\hat{Z}, 
\end{displaymath}
where $\tilde{Q} = Q_{1}\cdots Q_{k-1}$, and so on.  Now we break the transformation into three parts:
\begin{displaymath}
\tilde{A} - \lambda \tilde{B} = \tilde{Q}^{*}(A - \lambda B)\tilde{Z},
\end{displaymath}
\begin{equation}\label{eq:ktrans}
\check{A} - \lambda \check{B} = Q_{k}^{*}(\tilde{A} - \lambda\tilde{B})Z_{k-1},
\end{equation}
and
\begin{displaymath}
\hat{A} - \lambda \hat{B} = \hat{Q}^{*}(\check{A} - \lambda \check{B})\hat{Z}.
\end{displaymath}
Because each of the cores $Q_{1}$, \ldots, $Q_{k-1}$ leaves $\eee_{k}$ invariant,
we have 
\begin{displaymath}
\tilde{Q}\eee_{k} = \eee_{k} = r(AB^{-1})\eee_{k}, \quad\mbox{where $r(z) = 1$.}
\end{displaymath}
We can apply Theorem~\ref{thm:movespace} with $j=k$ to the transformation (\ref{eq:ktrans}), 
taking into account that the poles that are swapped in the $k$th 
move are $\rho$ and $\sigma_{k}$, to get
\begin{displaymath}
Q_{k}\eee_{k} = (\tilde{A} - \rho \tilde{B})(\tilde{A} - \sigma_{k}\tilde{B})^{-1}\eee_{k} = 
s(\tilde{A}\tilde{B}^{-1})\eee_{k}, \quad \mbox{where $s(z) = (z - \rho)/(z - \sigma_{k})$}.
\end{displaymath}
Finally, noting that $Q_{k+1}$, \ldots, $Q_{n-1}$ all leave $\eee_{k}$ invariant, we have
\begin{displaymath}
\hat{Q}\eee_{k} = \eee_{k} = t(\check{A}\check{B}^{-1})\eee_{k}, \quad\mbox{where $t(z) = 1$.}
\end{displaymath}
Now, applying Lemma~\ref{lem:accumstep} to the product $Q = \tilde{Q}Q_{k}\hat{Q}$, we get
\begin{displaymath}
Q\eee_{k} = tsr (AB^{-1})\eee_{k} = s(AB^{-1})\eee_{k},
\end{displaymath}
which is exactly (\ref{eq:kstep}).  

We can prove the $Z$ part of Theorem~\ref{thm:basicalgspace} in exactly the same way.  
We have 
\begin{displaymath}
\tilde{Z}\eee_{k-1} = \eee_{k-1} = 1(B^{-1}A)\eee_{k-1},
\end{displaymath}
and by Theorem~\ref{thm:movespace} with $j=k$,
\begin{displaymath}
Z_{k-1}\eee_{k-1} = (\tilde{A} - \sigma_{k}\tilde{B})^{-1}(\tilde{A} - \rho \tilde{B})\eee_{k-1} 
= s(\tilde{B}^{-1}\tilde{A})\eee_{k-1},
\end{displaymath}
and finally 
\begin{displaymath}
\hat{Z}\eee_{k-1} = \eee_{k-1} = 1(\check{B}^{-1}\check{A})\eee_{k-1}.
\end{displaymath}
Therefore, by Lemma~\ref{lem:accumstep},  
\begin{displaymath}
Z\eee_{k-1} = s(B^{-1}A)\eee_{k-1} = (A - \sigma_{k}B)^{-1}(A - \rho B)\eee_{k-1}.
\end{displaymath}
Adding one to the index $k$, we get the $Z$ part of Theorem~\ref{thm:basicalgspace}, thereby completing the proof.

\subsection*{Generalization of the proof}
The basic algorithm is just one of many possible algorithms that make use of moves of types I and II
on proper Hessenberg forms.  We have already pointed out that one could run the algorithm in the opposite
direction or in both directions at once.  There are lots of other possibilities, and we will look at 
some in what follows.  

From  our proof of Theorem~\ref{thm:basicalgspace} it should now be clear that we will be able to use 
Theorem~\ref{thm:movespace}, together with Lemma~\ref{lem:accumstep}, to analyze the action of 
any algorithm that acts on a proper Hessenberg pencil by moves of types I and II.  Consider a transformation
\begin{equation}\label{eq:generalg}
\hat{A} - \lambda \hat{B} = Q^{*}(A - \lambda B)Z,
\end{equation}
where $Q$ and $Z$ are products of core transformations generated by any sequence of moves of type I and II.
If we want to find the action of $Q$ on $\eee_{k}$ for some $k$, we need only look at the core transformations 
of the form $Q_{k}$, i.e.\ the ones that act in the $(k,k+1)$ plane.  Thus we factor $Q$ into a product
of the form 
\begin{equation}\label{eq:qfactlong}
Q = \tilde{Q}Q_{1,k}\check{Q}Q_{2,k}\hat{Q}Q_{3,k} \cdots , 
\end{equation}
where $\tilde{Q}$, $\check{Q}$, \ldots\, are products of core transformations that do not act in the $(k,k+1)$ plane
and therefore satisfy $\tilde{Q}\eee_{k} = \eee_{k}$, $\check{Q}\eee_{k} = \eee_{k}$, and so on, and $Q_{1,k}$, 
$Q_{2,k}$, \ldots\, are cores that do act in the $(k,k+1)$ plane.   Let us say there are $m$ such cores 
$Q_{1,k}$, \ldots, $Q_{m,k}$.

The transforming matrix $Z$ has a fully analogous factorization 
\begin{equation}\label{eq:zfactlong}
Z = \tilde{Z}Z_{1,k-1}\check{Z}Z_{2,k-1}\hat{Z}Z_{3,k-1} \cdots , 
\end{equation}
assuming we use the convention that moves of type I have the form $Q_{1}^{*}(A - \lambda B)Z_{0}$  with 
$Z_{0} = I$ or  $Q_{n}^{*}(A - \lambda B)Z_{n-1}$ with $Q_{n}=I$.  We have $\tilde{Z}\eee_{k-1} = \eee_{k-1}$, 
$\check{Z}\eee_{k-1} = \eee_{k-1}$, et cetera.   The transformations that act nontrivially on $\eee_{k-1}$ are
$Z_{1,k-1}$, \ldots, $Z_{m,k-1}$.

Suppose that on the move corresponding to the transformations $Q_{j,k}$ and $Z_{j,k-1}$, the poles that get 
swapped  are $\sigma_{j,k-1}$ and $\sigma_{j,k}$.  Then, according to Theorem~\ref{thm:movespace}, 
the function associated with this swap is 
$r_{j}(z) = (z - \sigma_{j,k-1})/(z -\sigma_{j,k})$.   Let $r$ denote the product of these functions: 
\begin{equation}\label{eq:ratrdef}
r(z) = r_{1}(z) \cdots r_{m}(z) = \prod_{j=1}^{m}\ \frac{z- \sigma_{j,k-1}}{z - \sigma_{j,k\phantom{-1}}}.
\end{equation}
Then, applying Lemma~\ref{lem:accumstep} to the long product of transformations defined by (\ref{eq:qfactlong}) 
and (\ref{eq:zfactlong}), we find that the action of $Q$ on $\eee_{k}$ and of $Z$ on $\eee_{k-1}$ is given
by 
\begin{equation}\label{eq:genact}
\cue_{k} = Q\eee_{k} = r(AB^{-1}) \eee_{k} \quad\mbox{and}\quad \zee_{k-1} = Z\eee_{k-1} = r(B^{-1}A)\eee_{k-1}.
\end{equation}

We summarize these findings as a theorem.

\begin{theorem}\label{thm:generalg}
Consider a transformation (\ref{eq:generalg}), where $Q$ and $Z$ are products of core transformations 
generated by any sequence of moves of types I and II.   For some $k$ suppose that $m$ of the moves acted
at the $k$th position, swapping poles $\sigma_{j,k-1}$ and $\sigma_{j,k}$ for  $j=1$, \ldots, $m$.  Define a rational 
function $r$ by (\ref{eq:ratrdef}).  Then the action of $Q$ on $\eee_{k}$ and of $Z$ on $\eee_{k-1}$ is given by 
(\ref{eq:genact}).   The transformation (\ref{eq:generalg}) transforms $\cue_{k}$ back to $\eee_{k}$ and 
$\zee_{k-1}$ back to $\eee_{k-1}$. 
\end{theorem}

\section{Variations on the basic algorithm} \label{sec:variations} 

In this section we consider algorithms built exclusively from moves of types I and II.   
Since the moves are backward stable, the resulting algorithms are also backward stable.  
We do not claim that all of the ideas presented here will result in practical algorithms;  
some of them are quite speculative.

The basic algorithm  (like the single-shift bulge-chasing and core-chasing algorithms)
takes a single shift, inserts it into the top of the pencil, and chases it to the bottom.  
This algorithm suffers from inefficient use of cache memory and negligible potential for 
parallelism.   In the case of bulge-chasing algorithms the problem was remedied by 
selecting a large number of shifts at once, creating many small bulges one after the
other, and chasing this chain of bulges together to the bottom of the matrix or 
pencil   \cite{BrByMa01,Lang97,Lang98}.   This allows the use of Level 3 BLAS and
therefore efficient cache use.  It also provides an opportunity for parallelism \cite{GrKaKr10}.

\subsection*{Chasing multiple shifts at once}
The same remedy works for pole-swapping algorithms, as was already mentioned in 
\cite{Cam19,CaMeVa19a,CaMeVa19b}.  We can choose $m$ shifts $\rho_{1}$, \ldots, $\rho_{m}$, where typically
$1 \ll m \ll n$.\footnote{One way to obtain $m$ shifts is to use an auxiliary routine to compute
the eigenvalues of the lower-right-hand $m \times m$ subpencil of $A - \lambda B$, and use
these as the shifts.}  Suppose the poles of $A - \lambda B$ are 
\begin{displaymath}
\sigma_{1}, \, \ldots,\, \sigma_{m},\, \sigma_{m+1}, \, \ldots,\, \sigma_{n}.
\end{displaymath}
By a sequence of moves of types I and II we can replace $\sigma_{1}$, \ldots, $\sigma_{m}$ 
by $\rho_{1}$, \ldots, $\rho_{m}$, so that the poles of the new pencil are
\begin{displaymath}
\rho_{1}, \, \ldots,\, \rho_{m},\, \sigma_{m+1}, \, \ldots,\, \sigma_{n}.
\end{displaymath}
Then we can chase these $m$ shifts together to the bottom, creating enough arithmetic to make 
efficient use of cache.  To be precise, in the first step we would swap $\sigma_{m+1}$ with $\rho_{m}$,
then $\sigma_{m+1}$ with $\rho_{m-1}$, and so on.  Eventually we swap $\sigma_{m+1}$ with $\rho_{1}$, putting
$\sigma_{m+1}$ at the top.  Then we go on to the next step.

We can pass a chain of shifts from top to bottom, and we can equally well pass a chain from bottom
to top.  If we wish, we can pass chains in both directions at once.  Suppose we have shifts $\rho_{1}$, \ldots,
$\rho_{m}$ that we wish to chase from top to bottom and shifts $\tau_{1}$, \ldots, $\tau_{m}$ that we wish 
to chase from bottom to top.  Using moves of types I and II we can introduce them:
\begin{displaymath}
\rho_{1},\, \ldots,\, \rho_{m}, \sigma_{m+1}, \, \ldots \, \sigma_{n-m-1}, \, \tau_{1}, \, \ldots, \, \tau_{m}.
\end{displaymath}
We then chase the $\rho$'s downward and the $\tau$'s upward.  The two chains pass through each other, 
and eventually we get to the position 
\begin{displaymath}
\tau_{1},\, \ldots,\, \tau_{m}, \sigma_{m+1}, \, \ldots \, \sigma_{n-m-1}, \, \rho_{1}, \, \ldots, \, \rho_{m}.
\end{displaymath}
The reader can check that the poles in the middle, $\sigma_{m+1}$, \ldots, $\sigma_{n-m-1}$, get moved 
around in the process, but they end up exactly where they started.  At this point we can regard the iteration
as complete, or we can ``complete'' the iteration by removing the $\tau_{i}$ and $\rho_{i}$ from the pencil 
and replacing them with new sets of shifts.  

Let's see what Theorem~\ref{thm:generalg} tells us about this bi-directional procedure.  Let
\begin{displaymath}
r(z) = \prod_{i=1}^{m}\frac{z - \rho_{i}}{z - \tau_{i}}.
\end{displaymath}
Then for $k=m+1$, \ldots, $n-m$ we have the action
\begin{displaymath}
\cue_{k} = Q\eee_{k} = r(AB^{-1}) \eee_{k} \quad\mbox{and}\quad \zee_{k-1} = Z\eee_{k-1} = r(B^{-1}A)\eee_{k-1}.
\end{displaymath}
The reason for this is that each of the $\rho_{i}$ passes downward through the $k$th position, causing a factor 
$z -\rho_{i}$, and each of the $\tau_{i}$ passes upward, causing a factor $(z -\tau_{i})^{-1}$.  This isn't all that 
happens at position $k$, but it's all that matters.  To see this, consider, for example, a position $k$ at which all
of the $\rho_{i}$ pass through before any of the $\tau_{i}$ get there.  Passing each $\rho_{i}$ downward requires
also passing a $\sigma_{j}$ upward, causing a factor $(z - \sigma_{j})^{-1}$.  Later on, when the $\tau_{i}$ are 
being passed upward, each $\sigma_{j}$ that was previously passed upward 
gets passed downward through the $k$th position, causing a factor
$z - \sigma_{j}$.  The factors $(z - \sigma_{j})^{-1}$ and $z -\sigma_{j}$ cancel each other out.  We know that this
must happen for each $\sigma_{j}$ because each $\sigma_{j}$ starts and ends in the same position. 

\subsection*{An optimistic scenario}
Consider a situation in which we have in hand the information
that we need to split the problem.  Suppose we know a $k$  (with $m+1 \leq k \leq n-m-1$) 
where (we think) we can split the pencil, and suppose that we have in mind 
an $(m,m)$  rational function 
\begin{displaymath}
r(z) = \prod_{i=1}^{m}\frac{z - \rho_{i}}{z - \tau_{i}}
\end{displaymath}
that can (nearly) split it.  By this we mean that $r(AB^{-1})\eee_{k}$ 
is (nearly) invariant under $AB^{-1}$ and $r(B^{-1}A)\eee_{k}$ is (nearly) invariant
under $B^{-1}A$.  If we then take the $\rho_{i}$ as shifts to be passed downward and the $\tau_{i}$
as shifts to be passed upward, we will get both $\cue_{k} = Q\eee_{k} = r(AB^{-1})\eee_{k}$ and 
$\zee_{k} = Z\eee_{k} = r(B^{-1}A)\eee_{k}$.  The change of variables 
$\hat{A} - \lambda \hat{B} = Q^{*}(A - \lambda B)Z$ maps both of these spaces back to $\eee_{k}$.  
Thus $\eee_{k}$ is (nearly) invariant under both $\hat{A}\hat{B}^{-1}$ and $\hat{B}^{-1}\hat{A}$, 
which implies that $(\eee_{k},\eee_{k})$ is (nearly) a deflating subspace for $(\hat{A},\hat{B})$.  
If the pencil does not quite split apart, another step with the same (or improved?) shifts may get 
the job done.  Notice that to achieve the desired spaces $\cue_{k} = r(AB^{-1})\eee_{k}$ and 
$\zee_{k} = r(B^{-1}A)\eee_{k}$, it is not necessary to pass the shifts all the way through the pencil.
All that is needed is that $\rho_{1}$, \ldots, $\rho_{m}$ are pushed downward past position $k+1$
and $\tau_{1}$, \ldots, $\tau_{m}$ are passed upward past position~$k$.   

Of course this is a very optimistic scenario.  (Where do we get these special shifts?)  
We include it here just to indicate what might be possible and to illustrate the use of 
Theorem~\ref{thm:generalg}.   

\subsection*{Practical shift strategies}
A more realistic plan is to take (for example) $\rho_{1}$, \ldots, $\rho_{m}$ 
to be the eigenvalues of the lower-right-hand $m\times m$ subpencil and $\tau_{1}$, \ldots, $\tau_{m}$ the 
eigenvalues of the upper-left-hand $m\times m$  subpencil, which will have the effect of causing deflations 
near the ends of the pencil.\footnote{Notice, however, that a strategy like this should also include some provision to ensure
that the upward-moving shifts are well separated from the downward-moving shifts.  
If some $\rho_{j}$ is (nearly) equal to one of the $\tau_{i}$, they will (nearly) cancel each other out.}  
An even better idea is to include \emph{aggressive early deflation} \cite{BrByMa01b}, which is easy to implement in this
context.   This was already discussed in detail in \cite{Cam19,CaMeVa19a}, so we will not go into it.

\subsection*{Steady streams of shifts}
We conclude this section with one more interesting but fanciful idea.  Imagine that we introduce 
steady streams of shifts at the top and the bottom.  Eventually the streams start to pass through
each other.  How do we move the streams in their respective directions in an expeditious way?  
To answer this question let us first look at the small case $n=8$, for which we have seven poles.
Suppose we have at some point the poles
\begin{displaymath}
\tau_{1},\ \rho_{3},\ \tau_{2},\ \rho_{2},\ \tau_{3},\ \rho_{1},\ \tau_{4},
\end{displaymath}
where the shifts $\rho_{i}$ are moving downward and the $\tau_{i}$ upward.  We can introduce a new
shift $\rho_{4}$ at the top by a move of type I that removes $\tau_{1}$.  At the same time we can do 
three moves of type II to interchange $\rho_{3}$ with $\tau_{2}$, $\rho_{2}$ with $\tau_{3}$, and $\rho_{1}$
with $\tau_{4}$.  The result is 
\begin{displaymath}
\rho_{4},\ \tau_{2},\ \rho_{3},\ \tau_{3},\ \rho_{2},\ \tau_{4},\ \rho_{1}.
\end{displaymath}
This is one step.  For the next step we use a move of type I to introduce a new shift $\tau_{5}$ at the bottom,
removing $\rho_{1}$.  At the same time we do three moves of type II to interchange  $\tau_{4}$ with $\rho_{2}$,   
$\tau_{3}$ with $\rho_{3}$, and $\tau_{2}$ with $\rho_{4}$.   The result is
\begin{displaymath}
\tau_{2},\ \rho_{4},\ \tau_{3},\ \rho_{3},\ \tau_{4},\ \rho_{2},\ \tau_{5}.
\end{displaymath}
The third step is like the first, the fourth step is like
the second, and so on.  We can illustrate these steps schematically with a diagram.  
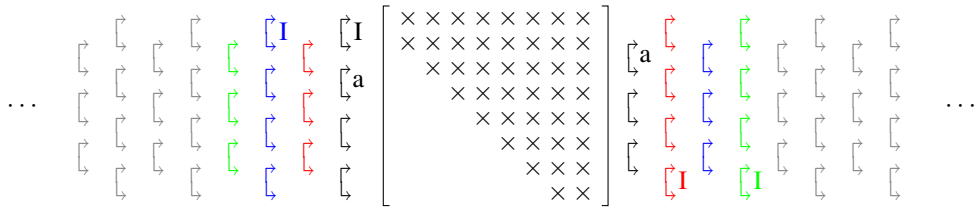
\begin{displaymath}
\cdots\quad
\parbox{11.6cm}{
\begin{tikzpicture}[scale=1.66,y=-1cm]
\tikzrotation[black][I]{-.5}{0}
\tikzrotation[black][a]{-.5}{.4}
\foreach \j in {2,...,3}{\tikzrotation[black]{-.5}{2*\j/5}}
\foreach \j in {0,...,2}{\tikzrotation[red]{-.8}{2*\j/5+0.2}}
\tikzrotation[blue][I]{-1.1}{0}
\foreach \j in {1,...,3}{\tikzrotation[blue]{-1.1}{2*\j/5}}
\foreach \j in {0,...,2}{\tikzrotation[green]{-1.4}{2*\j/5+0.2}}
\foreach \j in {0,...,3}{\tikzrotation[gray]{-1.7}{2*\j/5}}
\foreach \j in {0,...,2}{\tikzrotation[gray]{-2.0}{2*\j/5+0.2}}
\foreach \j in {0,...,3}{\tikzrotation[gray]{-2.3}{2*\j/5}}
\foreach \j in {0,...,2}{\tikzrotation[gray]{-2.6}{2*\j/5+0.2}}
\draw (-.15,-.1) -- (-.2,-.1) -- (-.2,1.5) -- (-.15,1.5);
\draw (1.55,-.1) -- (1.6,-.1) -- (1.6,1.5) -- (1.55,1.5);
\foreach \j in {0,...,7}{
   \foreach \i in {\j,...,7}{\node at (\i/5,\j/5)
     [align=center,scale=1.0]{$\times$};}}
\foreach \j in {0,...,6}{\node at (\j/5,\j/5+.2)
     [align=center,scale=1.0]{$\times$};}
\tikzrotation[black][a]{1.8}{0.2}
\foreach \j in {1,...,2}{\tikzrotation[black]{1.8}{2*\j/5+0.2}}
\foreach \j in {0,...,2}{\tikzrotation[red]{2.1}{2*\j/5}}
\tikzrotation[red][I]{2.1}{1.2}
\foreach \j in {0,...,2}{\tikzrotation[blue]{2.4}{2*\j/5+0.2}}
\foreach \j in {0,...,2}{\tikzrotation[green]{2.7}{2*\j/5}}
\tikzrotation[green][I]{2.7}{1.2}
\foreach \j in {0,...,2}{\tikzrotation[gray]{3.0}{2*\j/5+0.2}}
\foreach \j in {0,...,3}{\tikzrotation[gray]{3.3}{2*\j/5}}
\foreach \j in {0,...,2}{\tikzrotation[gray]{3.6}{2*\j/5+0.2}}
\foreach \j in {0,...,3}{\tikzrotation[gray]{3.9}{2*\j/5}}
\end{tikzpicture}
}\cdots
\end{displaymath}
The matrix in the middle can be either $A$ or $B$, since the same core transformations
are applied to both.  The cores of the first step are in black, with the move of type I marked 
accordingly.  Each move of type II requires two cores, one on the left and one on the right.  
For example, the two cores marked with the symbol ``$a$'' belong to a single move.  
The second step is in red, with the core of type I marked accordingly on the right.  
The third and fourth steps are marked in blue and green, respectively.  Four subsequent steps are
shown in grey.  We are illustrating the case $n=8$, which is typical of even $n$.  
The odd case, which is slightly different, is left for the reader.  

Before we get too excited about this elegant scheme, we must acknowledge that there are
some challenges in the way of a competitive implementation.  
Thinking now of larger $n$, we see that each step is rich in arithmetic and highly parallel.
Each step consists of about $n/2$ moves or $O(n^{2})$ flops.  To move a shift from one 
end of the pencil to the other requires about $n$ steps, or $O(n^{3})$ flops.  Therefore 
any competitive implementation must exploit the parallelism well.  Another, possibly larger,
issue is this:  How do we get a steady stream of good shifts to feed in at the two ends?

\section{Connections to earlier work}
\label{sec:connections}

\subsection*{Bulge pencils}
The purpose of shifting is to accelerate convergence.  In the
standard Francis bulge-chasing algorithm the shifts are inserted at
the top.  That is, the shifts are used to help determine the initial
transformation that creates the bulge.  Then the shifts are forgotten,
and the bulge is chased downward until it disappears off the bottom.
Well-chosen shifts, inserted at the top, lead to rapid emergence of
eigenvalues at the bottom of the matrix or pencil. Thus the
information about the shifts is somehow transmitted in the bulge from
top to bottom.  

A bit more than twenty years ago one of the authors
began to study the mechanism by which the shift information is conveyed in bulge-chasing algorithms. 
This study took some time, it seemed to be nontrivial, and it led to the discovery of the \emph{bulge pencil} \cite{Wat95,Wat96,Wat07}.    

Now let's take a fresh look at the bulge pencil in light of what we now know about pole swapping.  
Suppose we pick a single shift $\rho$ and begin chasing a bulge downward in a pair $(A,B)$.  After couple of
steps we have
\begin{equation}\label{eq:bp:1}
\parbox{2.cm}{
\begin{tikzpicture}[scale=1.66,y=-1cm]
\draw (-.15,-.1) -- (-.2,-.1) -- (-.2,1.1) -- (-.15,1.1);
\draw (1.15,-.1) -- (1.2,-.1) -- (1.2,1.1) -- (1.15,1.1);
\foreach \j in {0,...,5}{
   \foreach \i in {\j,...,5}{\node at (\i/5,\j/5)
     [align=center,scale=1.0]{$\times$};}}
\foreach \j in {0,...,4}{\node at (\j/5,\j/5+.2)
     [align=center,scale=1.0]{$\times$};}
   \node[red] at (0.2,.6){$+$};
   \draw[densely dotted] (.1,.7) -- (.5,.7) -- (.5, 0.3) -- (.1,0.3)--cycle;

\end{tikzpicture} 
} \qquad \qquad
\parbox{2.cm}{
\begin{tikzpicture}[scale=1.66,y=-1cm]
\draw (-.15,-.1) -- (-.2,-.1) -- (-.2,1.1) -- (-.15,1.1);
\draw (1.15,-.1) -- (1.2,-.1) -- (1.2,1.1) -- (1.15,1.1);
\foreach \j in {0,...,5}{
   \foreach \i in {\j,...,5}{\node at (\i/5,\j/5)
     [align=center,scale=1.0]{$\times$};}}
 \draw[densely dotted] (.1,.7) -- (.5,.7) -- (.5, 0.3) -- (.1,0.3)--cycle;
\end{tikzpicture} 
} \qquad,
\end{equation}
with the bulge located at position $(4,2)$. The  $2 \times 2$ subpencil outlined in (\ref{eq:bp:1}) 
is the bulge pencil.   Its eigenvalues 
are $\rho$ and $\infty$ \cite[Chap.~7]{Wat07}.\footnote{In \cite{Wat07} we considered (large-bulge) 
multishift algorithms with $k$ shifts $\rho_{1}$, \ldots, $\rho_{k}$.  Then the bulge pencil is $(k+1) \times (k+1)$
and has eigenvalues $\rho_{1}$, \ldots, $\rho_{k}$, and $\infty$.
Here we are considering only the case $k=1$.}   If we now do one more transformation on the left, 
moving the bulge from $A$ to $B$,  we obtain 
\begin{equation*}
\parbox{2.cm}{
\begin{tikzpicture}[scale=1.66,y=-1cm]
\draw (-.15,-.1) -- (-.2,-.1) -- (-.2,1.1) -- (-.15,1.1);
\draw (1.15,-.1) -- (1.2,-.1) -- (1.2,1.1) -- (1.15,1.1);
\foreach \j in {0,...,5}{
   \foreach \i in {\j,...,5}{\node at (\i/5,\j/5)
     [align=center,scale=1.0]{$\times$};}}
\foreach \j in {0,...,4}{\node at (\j/5,\j/5+.2)
     [align=center,scale=1.0]{$\times$};}
   \draw[densely dotted] (.1,.7) -- (.5,.7) -- (.5, 0.3) -- (.1,0.3)--cycle;

\end{tikzpicture} 
} \qquad \qquad
\parbox{2.cm}{
\begin{tikzpicture}[scale=1.66,y=-1cm]
\draw (-.15,-.1) -- (-.2,-.1) -- (-.2,1.1) -- (-.15,1.1);
\draw (1.15,-.1) -- (1.2,-.1) -- (1.2,1.1) -- (1.15,1.1);
\foreach \j in {0,...,5}{
   \foreach \i in {\j,...,5}{\node at (\i/5,\j/5)
     [align=center,scale=1.0]{$\times$};}}
 \draw[densely dotted] (.1,.7) -- (.5,.7) -- (.5, 0.3) -- (.1,0.3)--cycle;
\node[red] at (.4,.6){$+$};
\end{tikzpicture} 
} \qquad.
\end{equation*}
This is a Hessenberg pair, and the eigenvalues of the bulge pencil are now in plain sight.
In the $(3,2)$ position we have the pole $\infty$, and in the $(4,3)$ position we have 
a finite pole, which we know to be the shift $\rho$.  What was opaque before is now transparent.  

Certain structured problems require algorithms that chase bulges in both directions in 
order to preserve the structure.  The first example of such an algorithm was the 
Hamiltonian QR algorithm of Byers \cite{Bye83,Bye86}.  Some more recent examples
are algorithms for 
the palindromic and even eigenvalue problems discussed in \cite{KrScWa08}.
Our understanding of the bulge pencil made it possible to explain completely how
to pass bulges (and the shifts that they contain) through each other in general in both structured and unstructured
cases \cite{Wat98}.  It took time and effort figure this out, but now, in light of what 
we know about pole swapping, we can see that passing shifts through each other
is simple.  It's just a matter of swapping two eigenvalues of the pole pencil.
Once again, what was opaque before is now transparent.  

\subsection*{Tightly and optimally packed shifts}
The schemes discussed in Section~\ref{sec:variations} insert not just one shift but long chains of shifts 
$\rho_{1}$, \ldots, $\rho_{m}$ into the pencil as poles and then chase them 
downward (or upward) in a bunch.  In such a scheme it is important for efficiency 
to have the shifts  packed as tightly together as possible.  It is clear that in our current 
scenario we achieve this; the shifts $\rho_{1}$, \ldots, $\rho_{m}$ appear as adjacent poles 
in the Hessenberg pair, and there is no way they could be packed any closer.  
(The same result is achieved effortlessly when this methodology
is applied to core-chasing algorithms \cite{AuMaRoVaWa18}.)  In contrast, in the
bulge-chasing scenario, the packing of bulges is not naturally optimal, and it is not
obvious how to fix the problem. However, with some effort a remedy was eventually found
\cite{KaKrLa14}.  In hindsight we can show that the remedy is a disguised implementation of a 
pole-swapping algorithm. 

We have explained already that pole swapping reduces to bulge
chasing if all poles that are not shifts are set to infinity. 
The philosophy is, however, different. 
Bulge chasing executes in each
step an equivalence where the transforms on left and right act on columns and rows
having the same indices, say $i$ and $i+1$.
Pole swapping, on the other hand, has the transformation on the left acting on rows $i$ and $i+1$, 
while the transformation on the right acts on columns $i-1$ and $i$.
Pole swapping is half an equivalence off compared to bulge chasing. This lag is
natural in the pole-swapping setting and appears to be the foundational strategy to get
optimally packed bulges.

An optimally packed chain of two single shifts in the bulge chasing setting would, ideally, look
like
\begin{equation}
  \label{eq:optpack:1}
  \parbox{2.cm}
  {
\begin{tikzpicture}[scale=1.66,y=-1cm]
\draw (-.15,-.1) -- (-.2,-.1) -- (-.2,1.1) -- (-.15,1.1);
\draw (1.15,-.1) -- (1.2,-.1) -- (1.2,1.1) -- (1.15,1.1);
\foreach \j in {0,...,5}{
   \foreach \i in {\j,...,5}{\node at (\i/5,\j/5)
     [align=center,scale=1.0]{$\times$};}}
\foreach \j in {0,...,4}{\node at (\j/5,\j/5+.2)
     [align=center,scale=1.0]{$\times$};}
{\color{red}\node at (0,.4){$+$};\color{black}}
{\color{red}\node at (.2,.6){$+$};\color{black}}
\end{tikzpicture} 
} \qquad \qquad
\parbox{2.cm}
{
\begin{tikzpicture}[scale=1.66,y=-1cm]
\draw (-.15,-.1) -- (-.2,-.1) -- (-.2,1.1) -- (-.15,1.1);
\draw (1.15,-.1) -- (1.2,-.1) -- (1.2,1.1) -- (1.15,1.1);
\foreach \j in {0,...,5}{
   \foreach \i in {\j,...,5}{\node at (\i/5,\j/5)
     [align=center,scale=1.0]{$\times$};}}
\phantom{\color{red}\node at (.2,.6){$+$};\color{black}}
\end{tikzpicture} 
} \qquad ,
\end{equation}
whereas in the pole-swapping setting it would resemble
\begin{equation*}\label{eq:optpack:2}
\parbox{2.cm}{
\begin{tikzpicture}[scale=1.66,y=-1cm]
\draw (-.15,-.1) -- (-.2,-.1) -- (-.2,1.1) -- (-.15,1.1);
\draw (1.15,-.1) -- (1.2,-.1) -- (1.2,1.1) -- (1.15,1.1);
\foreach \j in {0,...,5}{
   \foreach \i in {\j,...,5}{\node at (\i/5,\j/5)
     [align=center,scale=1.0]{$\times$};}}
\foreach \j in {0,...,4}{\node at (\j/5,\j/5+.2)
     [align=center,scale=1.0]{$\times$};}
\phantom{\color{red}\node at (0,.4){$+$};\color{black}}
\phantom{\color{red}\node at (.2,.6){$+$};\color{black}}
\end{tikzpicture} 
} \qquad \qquad
\parbox{2.cm}{
\begin{tikzpicture}[scale=1.66,y=-1cm]
\draw (-.15,-.1) -- (-.2,-.1) -- (-.2,1.1) -- (-.15,1.1);
\draw (1.15,-.1) -- (1.2,-.1) -- (1.2,1.1) -- (1.15,1.1);
\foreach \j in {0,...,5}{
   \foreach \i in {\j,...,5}{\node at (\i/5,\j/5)
     [align=center,scale=1.0]{$\times$};}}
\color{red}\node at (0,.2){$+$};\color{black}
{\color{red}\node at (.2,.4){$+$};\color{black}}
\end{tikzpicture} 
} \qquad.
\end{equation*}
For simplicity, and without loss of generality,
we restrict ourselves to two single shifts.

We have seen that getting an optimally packed chain of shifts in the
pole-swapping setting is trivial.
In the bulge chasing case, however, it is impossible to achieve \eqref{eq:optpack:1}. Introducing the first shift and chasing it down a
row results in 
\begin{equation*}
  \parbox{2.cm}
  {
\begin{tikzpicture}[scale=1.66,y=-1cm]
\draw (-.15,-.1) -- (-.2,-.1) -- (-.2,1.1) -- (-.15,1.1);
\draw (1.15,-.1) -- (1.2,-.1) -- (1.2,1.1) -- (1.15,1.1);
\foreach \j in {0,...,5}{
   \foreach \i in {\j,...,5}{\node at (\i/5,\j/5)
     [align=center,scale=1.0]{$\times$};}}
\foreach \j in {0,...,4}{\node at (\j/5,\j/5+.2)
     [align=center,scale=1.0]{$\times$};}
{\color{red}\node at (.2,.6){$+$};\color{black}}
\end{tikzpicture} 
} \qquad \qquad
\parbox{2.cm}
{
\begin{tikzpicture}[scale=1.66,y=-1cm]
\draw (-.15,-.1) -- (-.2,-.1) -- (-.2,1.1) -- (-.15,1.1);
\draw (1.15,-.1) -- (1.2,-.1) -- (1.2,1.1) -- (1.15,1.1);
\foreach \j in {0,...,5}{
   \foreach \i in {\j,...,5}{\node at (\i/5,\j/5)
     [align=center,scale=1.0]{$\times$};}}
\phantom{\color{red}\node at (.2,.6){$+$};\color{black}}
\end{tikzpicture} 
} \qquad .
\end{equation*}
Introducing the second shift does not work. We end up with
\begin{equation*}
  \parbox{2.cm}
  {
\begin{tikzpicture}[scale=1.66,y=-1cm]
\draw (-.15,-.1) -- (-.2,-.1) -- (-.2,1.1) -- (-.15,1.1);
\draw (1.15,-.1) -- (1.2,-.1) -- (1.2,1.1) -- (1.15,1.1);
\foreach \j in {0,...,5}{
   \foreach \i in {\j,...,5}{\node at (\i/5,\j/5)
     [align=center,scale=1.0]{$\times$};}}
\foreach \j in {0,...,4}{\node at (\j/5,\j/5+.2)
     [align=center,scale=1.0]{$\times$};}
   \node[red] at (0,.4){$+$};
   \node[red] at (0,.6){$+$};
   \node[red] at (.2,.6){$+$};
\end{tikzpicture} 
} \qquad \qquad
\parbox{2.cm}
{
\begin{tikzpicture}[scale=1.66,y=-1cm]
\draw (-.15,-.1) -- (-.2,-.1) -- (-.2,1.1) -- (-.15,1.1);
\draw (1.15,-.1) -- (1.2,-.1) -- (1.2,1.1) -- (1.15,1.1);
\foreach \j in {0,...,5}{
   \foreach \i in {\j,...,5}{\node at (\i/5,\j/5)
     [align=center,scale=1.0]{$\times$};}}
\phantom{\color{red}\node at (.2,.6){$+$};\color{black}}
\end{tikzpicture}
} \qquad ,
\end{equation*}
and  both single shifts have been combined into a $2\times 2$ multishift bulge.
The scheme introduced by Braman, Byers, and Mathias \cite{p522} delays the introduction 
of the second shift until the first has been moved two spots down. We get
\begin{equation*}
  \parbox{2.cm}
  {
\begin{tikzpicture}[scale=1.66,y=-1cm]
\draw (-.15,-.1) -- (-.2,-.1) -- (-.2,1.1) -- (-.15,1.1);
\draw (1.15,-.1) -- (1.2,-.1) -- (1.2,1.1) -- (1.15,1.1);
\draw[densely dotted] (-.1,.5) -- (.3,.5) -- (.3, 0.1) -- (-.1,0.1)--cycle;
\draw[densely dotted] (.3,.9) -- (.7,.9) -- (.7, 0.5) -- (.3,0.5)--cycle;
\foreach \j in {0,...,5}{
   \foreach \i in {\j,...,5}{\node at (\i/5,\j/5)
     [align=center,scale=1.0]{$\times$};}}
\foreach \j in {0,...,4}{\node at (\j/5,\j/5+.2)
     [align=center,scale=1.0]{$\times$};}
   \node[red] at (0,.4){$+$};
   \node[red] at (.4,.8){$+$};
\end{tikzpicture} 
} \qquad \qquad
\parbox{2.cm}
{
\begin{tikzpicture}[scale=1.66,y=-1cm]
\draw (-.15,-.1) -- (-.2,-.1) -- (-.2,1.1) -- (-.15,1.1);
\draw (1.15,-.1) -- (1.2,-.1) -- (1.2,1.1) -- (1.15,1.1);
\foreach \j in {0,...,5}{
   \foreach \i in {\j,...,5}{\node at (\i/5,\j/5)
     [align=center,scale=1.0]{$\times$};}}
\phantom{\color{red}\node at (.2,.6){$+$};\color{black}}
\end{tikzpicture} 
} \qquad, 
\end{equation*}
which are so-called tightly packed shifts. It is impossible to pack
them any closer; otherwise the two $2\times 2$ bulge pencils
(marked in the figure) would overlap.

A solution to pack the bulges as tight as in (\ref{eq:optpack:1}) was proposed by
Karlsson, Kressner, and Lang \cite{KaKrLa14}. The trick is to defer some transformations from the right. Suppose the first bulge is introduced and we would like to move it down a row; instead of
executing an entire bulge-chasing step, we only execute the transformation from the left,
the transformation on the right is postponed. We end up with
\begin{equation*}
\parbox{2.cm}{
\begin{tikzpicture}[scale=1.66,y=-1cm]
\draw (-.15,-.1) -- (-.2,-.1) -- (-.2,1.1) -- (-.15,1.1);
\draw (1.15,-.1) -- (1.2,-.1) -- (1.2,1.1) -- (1.15,1.1);
\foreach \j in {0,...,5}{
   \foreach \i in {\j,...,5}{\node at (\i/5,\j/5)
     [align=center,scale=1.0]{$\times$};}}
\foreach \j in {0,...,4}{\node at (\j/5,\j/5+.2)
     [align=center,scale=1.0]{$\times$};}
\phantom{\color{red}\node at (.2,.6){$+$};\color{black}}
\end{tikzpicture} 
} \qquad \qquad
\parbox{2.cm}{
\begin{tikzpicture}[scale=1.66,y=-1cm]
\draw (-.15,-.1) -- (-.2,-.1) -- (-.2,1.1) -- (-.15,1.1);
\draw (1.15,-.1) -- (1.2,-.1) -- (1.2,1.1) -- (1.15,1.1);
\foreach \j in {0,...,5}{
   \foreach \i in {\j,...,5}{\node at (\i/5,\j/5)
     [align=center,scale=1.0]{$\times$};}}
\node[red] at (.2,.4){$+$};
\end{tikzpicture} 
} \qquad, 
\end{equation*}
which is nothing else than having moved the first pole down a position. Next we introduce
the second shift, but we do not execute the transformation from the right. We get
\begin{equation*}
\parbox{2.cm}{
\begin{tikzpicture}[scale=1.66,y=-1cm]
\draw (-.15,-.1) -- (-.2,-.1) -- (-.2,1.1) -- (-.15,1.1);
\draw (1.15,-.1) -- (1.2,-.1) -- (1.2,1.1) -- (1.15,1.1);
\foreach \j in {0,...,5}{
   \foreach \i in {\j,...,5}{\node at (\i/5,\j/5)
     [align=center,scale=1.0]{$\times$};}}
\foreach \j in {0,...,4}{\node at (\j/5,\j/5+.2)
     [align=center,scale=1.0]{$\times$};}
\phantom{\color{red}\node at (.2,.6){$+$};\color{black}}
\end{tikzpicture} 
} \qquad \qquad
\parbox{2.cm}{
\begin{tikzpicture}[scale=1.66,y=-1cm]
\draw (-.15,-.1) -- (-.2,-.1) -- (-.2,1.1) -- (-.15,1.1);
\draw (1.15,-.1) -- (1.2,-.1) -- (1.2,1.1) -- (1.15,1.1);
\foreach \j in {0,...,5}{
   \foreach \i in {\j,...,5}{\node at (\i/5,\j/5)
     [align=center,scale=1.0]{$\times$};}}
\node[red] at (0,.2){$+$};
\node[red] at (.2,.4){$+$};
\end{tikzpicture} 
} \qquad.
\end{equation*}
To start the chasing, one now brings the first shift to the right, creating a bulge and then annihlates the bulge. Thus, one does not execute an entire bulge-chase step, but again the transformation from the right is delayed. We end up with
\begin{equation*}
\parbox{2.cm}{
\begin{tikzpicture}[scale=1.66,y=-1cm]
\draw (-.15,-.1) -- (-.2,-.1) -- (-.2,1.1) -- (-.15,1.1);
\draw (1.15,-.1) -- (1.2,-.1) -- (1.2,1.1) -- (1.15,1.1);
\foreach \j in {0,...,5}{
   \foreach \i in {\j,...,5}{\node at (\i/5,\j/5)
     [align=center,scale=1.0]{$\times$};}}
\foreach \j in {0,...,4}{\node at (\j/5,\j/5+.2)
     [align=center,scale=1.0]{$\times$};}
\phantom{\color{red}\node at (.2,.6){$+$};\color{black}}
\end{tikzpicture} 
} \qquad \qquad
\parbox{2.cm}{
\begin{tikzpicture}[scale=1.66,y=-1cm]
\draw (-.15,-.1) -- (-.2,-.1) -- (-.2,1.1) -- (-.15,1.1);
\draw (1.15,-.1) -- (1.2,-.1) -- (1.2,1.1) -- (1.15,1.1);
\foreach \j in {0,...,5}{
   \foreach \i in {\j,...,5}{\node at (\i/5,\j/5)
     [align=center,scale=1.0]{$\times$};}}
\node[red] at (0,.2){$+$};
\node[red] at (.4,.6){$+$};
\end{tikzpicture} 
} \qquad,
\end{equation*}
after which we can do the same with the second shift.
Obviously this is just pole swapping, but the description in terms of bulges and
delayed transformations conceals this fact.

Karlsson et al.\ \cite{KaKrLa14} discussed the optimal packing of the bulges in terms of 
double-shift bulges.
Since we have not discussed double-shift pole-swapping algorithms here, we do not
explore this. The principles are, however, identical. 
The algorithm of Karlsson et al.\ \cite{KaKrLa14} is a pole-swapping algorithm (with poles at infinity) avant-la-lettre.

\section{The new pole-swapping procedure}\label{sec:newswap}

We now describe the new swapping procedure that was promised at the beginning. 
The process of swapping two adjacent poles is equivalent to swapping two adjacent eigenvalues in the 
upper-triangular pole pencil.  For the description it suffices to look at a $2 \times 2$ subpencil.   
Consider therefore a $2 \times 2$ upper-triangular pencil  
\begin{equation}\label{eq:startsubpencil}
A - \lambda B = \left[\begin{array}{cc} \alpha_{1} & a  \\ 0 & \alpha_{2}\end{array}\right] 
- \lambda \left[\begin{array}{cc} \beta_{1} & b  \\ 0 & \beta_{2}  \end{array}\right]
\end{equation}
with eigenvalues $\sigma_{1} = \alpha_{1}/\beta_{1}$ and $\sigma_{2} = \alpha_{2}/\beta_{2}$.  We want to swap the eigenvalues.  That is, we want to find core transformations $Q$ and $Z$ such that
\begin{displaymath}
Q^{*}(A - \lambda B)Z = \hat{A} - \lambda\hat{B} = \left[\begin{array}{cc} \hat{\alpha}_{1} & \hat{a} 
\\ & \hat{\alpha}_{2}\end{array}\right] - \lambda \left[\begin{array}{cc} \hat{\beta}_{1} & \hat{b} \\ 
& \hat{\beta}_{2}\end{array}\right], 
\end{displaymath}
with $\hat{\alpha}_{1}/\hat{\beta}_{1} = \sigma_{2}$ and $\hat{\alpha}_{2}/\hat{\beta}_{2} = \sigma_{1}$.   

\subsection*{Solution in exact arithmetic}

\subsubsection*{Exact method 1}  This method ``grabs $\sigma_{2}$'' and pulls it upward.
Substituting $\lambda = \alpha_{2}/\beta_{2}$ in the pencil, we have 
\begin{displaymath}
\beta_{2}A - \alpha_{2}B = \left[\begin{array}{cc} 
\beta_{2}\alpha_{1} - \alpha_{2}\beta_{1} & \beta_{2}a - \alpha_{2}b \\ 0 & 0
\end{array}\right],
\end{displaymath}
from which we deduce that the vector
\begin{equation}\label{eq:xdef}
x = \left[\begin{array}{c} \alpha_{2}b - \beta_{2}a \\
\beta_{2}\alpha_{1} - \alpha_{2}\beta_{1}
\end{array}\right]
\end{equation}
is a right eigenvector of the pencil associated with eigenvalue $\sigma_{2} = \alpha_{2}/\beta_{2}$.
Let 
\begin{equation}\label{eq:ydef}
y = \left[\begin{array}{c} \alpha_{1}b - \beta_{1}a \\
\beta_{2}\alpha_{1} - \alpha_{2}\beta_{1}
\end{array}\right].
\end{equation}
Direct computation shows that 
\begin{equation}\label{eq:xydeflate}
Ax = \alpha_{2}y \quad\mbox{and}\quad Bx = \beta_{2}y.
\end{equation}
Thus the spaces spanned by $x$ and $y$ form a one-dimensional deflating pair for $(A,B)$ associated with the
eigenvalue $\alpha_{2}/\beta_{2}$.

Let $Q$ and $Z$ be cores such that 
\begin{displaymath}
Z^{*}x = \gamma e_{1} \quad\mbox{and}\quad Q^{*}y = \zeta e_{1}, 
\end{displaymath}
and define 
\begin{displaymath}
\hat{A} - \lambda \hat{B} = Q^{*}AZ - \lambda\, Q^{*}BZ.
\end{displaymath}
Then we claim that $\hat{A} - \lambda \hat{B}$ is an upper triangular pencil with the eigenvalue $\alpha_{2}/\beta_{2}$ on top.
This is verified by the calculations
\begin{displaymath}
\hat{A}e_{1} = Q^{*}AZe_{1} = \gamma^{-1}Q^{*}Ax = \alpha_{2}\gamma^{-1}Q^{*}y = \alpha_{2}\gamma^{-1}\zeta\, e_{1}
\end{displaymath}
and
\begin{displaymath}
\hat{B}e_{1} = Q^{*}BZe_{1} = \gamma^{-1}Q^{*}Bx = \beta_{2}\gamma^{-1}Q^{*}y = \beta_{2}\gamma^{-1}\zeta\, e_{1}.
\end{displaymath}
This procedure fails if and only if $x=0$, which happens whenever $A = B$, for example.   The condition $x=0$ implies
that the eigenvalues are equal, so in this case the swap can be skipped.

\subsubsection*{Exact method 2}  This method, which is the dual of the previous method, 
``grabs $\sigma_{1}$'' and pushes it downward.  
Substituting $\lambda = \alpha_{1}/\beta_{1}$ in the pencil, we have 
\begin{displaymath}
\beta_{1}A - \alpha_{1}B = \left[\begin{array}{cc} 
0 & \beta_{1}a - \alpha_{1}b \\ 0 & \beta_{1}\alpha_{2} - \alpha_{1} \beta_{2}
\end{array}\right],
\end{displaymath}
from which we deduce that the vector
\begin{equation}\label{eq:vdef}
v^{T} = \left[\begin{array}{cc} \beta_{1}\alpha_{2} - \alpha_{1}\beta_{2} & \alpha_{1}b - \beta_{1} a 
\end{array}\right]
\end{equation}
is a left eigenvector of the pencil associated with eigenvalue $\sigma_{1} = \alpha_{1}/\beta_{1}$.
Let 
\begin{equation}\label{eq:wdef}
w^{T} = \left[\begin{array}{cc}  \beta_{1}\alpha_{2} - \alpha_{1}\beta_{2} & \alpha_{2}b - \beta_{2}a
\end{array}\right].
\end{equation}
Direct computation shows that 
\begin{equation}\label{eq:vwdeflate}
v^{T}A = \alpha_{1}w^{T} \quad\mbox{and}\quad v^{T}B = \beta_{1}w^{T}.
\end{equation}

Let $Q$ and $Z$ be cores such that 
\begin{displaymath}
v^{T}Q = \zeta e_{2}^{T} \quad\mbox{and}\quad w^{T}Z = \gamma e_{2}^{T},
\end{displaymath}
and define 
\begin{displaymath}
\hat{A} - \lambda \hat{B} = Q^{*}AZ - \lambda\, Q^{*}BZ.
\end{displaymath}
Then we claim that $\hat{A} - \lambda \hat{B}$ is an upper triangular pencil with the eigenvalue $\alpha_{1}/\beta_{1}$ on 
the bottom.
This is verified by the calculations
\begin{displaymath}
e_{2}^{T}\hat{A} = e_{2}^{T}Q^{*}AZ = \zeta^{-1}v^{T}AZ = \alpha_{1}\zeta^{-1}w^{T}Z = \alpha_{1}\zeta^{-1}\gamma\, e_{2}^{T}
\end{displaymath}
and
\begin{displaymath}
e_{2}^{T}\hat{B} = e_{2}^{T}Q^{*}BZ = \zeta^{-1}v^{T}BZ = \beta_{1}\zeta^{-1}w^{T}Z = \beta_{1}\zeta^{-1}\gamma\, e_{2}^{T}.
\end{displaymath}
This procedure fails if and only if $v^{T}=0$, in which case $\sigma_{1} = \sigma_{2}$ and the swap can be skipped.

The reader can  easily check that the two methods produce exactly the same $Q$ and $Z$.

\subsection*{Solution in floating point arithmetic}

In the interest of stability one should not implement either of the above procedures in practice.  
There are several alternatives.  

\subsubsection*{Case 1} 
We will demonstrate below that the following procedure, which is based on exact method 1, 
is stable in the case $\absval{\sigma_{1}} \geq \absval{\sigma_{2}}$. 
Compute $x$ as in (\ref{eq:xdef}).  Then compute $Z$ such that $Z^{*}x = \gamma e_{1}$, where 
$\gamma = \norm{x}$.  (Here and in what follows, the norm symbol refers to either the vector 
2-norm or matrix 2-norm, depending on the context.)
Then compute $BZ$.   Since $Ze_{1} = \gamma^{-1}x$, the first column of $BZ$ is 
$\gamma^{-1}\beta_{2}\,y$.  Do not compute $Q$  using the vector $y$ as defined in (\ref{eq:ydef}).  
Instead compute $Q$ so that $Q^{*}(BZe_{1}) = \beta_{2}\gamma^{-1}\zeta \, e_{1}$.
Then let 
\begin{displaymath}
\hat{A} = Q^{*}AZ \quad\mbox{and}\quad \hat{B} = Q^{*}BZ.
\end{displaymath} 

\subsubsection*{Case 2}
For the case when $\absval{\sigma_{1}} < \absval{\sigma_{2}}$ we need a different procedure.
There are multiple possibilities, the simplest of which is to apply the above procedure with the roles
of $A$ and $B$ reversed.  We compute $Z$ as before, then use $AZ$ instead of $BZ$ to determine $Q$.
Specifically, since $Ze_{1} = \gamma^{-1}x$, the first column of $AZ$ is 
$\gamma^{-1}\alpha_{2}\,y$.  Thus we can compute $Q$ so that $Q^{*}(AZe_{1}) = \alpha_{2}\gamma^{-1}\zeta \, e_{1}$.

This procedure is similar to that of Van Dooren \cite{VanD81}.   His method always computes $Z$ first, then uses
either $AZ$ or $BZ$ to compute $Q$.  The only difference is that our criterion for switching between
$BZ$ and $AZ$ is different from that in \cite{VanD81}.  This makes a difference in the backward error. 

Another procedure, which is based on exact method 2, computes $Q$ first.  Compute the vector $v^{T}$ as in 
(\ref{eq:vdef}), then compute $Q$ such that $v^{T}Q = \zeta e_{2}^{T}$, where $\zeta = \norm{v}$.  
Then compute $Q^{*}B$.   Since $e_{2}^{T}Q^{*} = \zeta^{-1} v^{T}$, the second row of $Q^{*}B$ is 
$\zeta^{-1}\beta_{1}\,w^{T}$.  Do not compute $Z$ using $w^{T}$ as defined in (\ref{eq:wdef}).  Instead compute $Z$ 
so that $(e_{2}^{T}Q^{*}B)Z = \beta_{1}\zeta^{-1}\gamma\, e_{2}^{T}$.  Then let 
\begin{displaymath}
\hat{A} = Q^{*}AZ \quad\mbox{and}\quad \hat{B} = Q^{*}BZ.
\end{displaymath} 

This is exactly equivalent to the procedure from Case 1 applied to a ``flipped'' pencil.  Let 
\smash{$F = \left[\begin{smallmatrix} 0 & 1 \\ 1 & 0\end{smallmatrix}\right]$}, the \emph{flip} matrix, and consider the 
pencil 
\begin{displaymath}
FA^{T}F - \lambda FB^{T}F = \left[\begin{array}{cc} \alpha_{2} & a \\ & \alpha_{1}\end{array}\right] -
\lambda \left[\begin{array}{cc} \beta_{2} & b \\ & \beta_{1} \end{array}\right].
\end{displaymath}
This has the eigenvalues reversed.  The condition $\absval{\sigma_{2}} > \absval{\sigma_{1}}$ implies that we can
stably apply the method from Case 1, and then ``unflip'' the result.  The equation 
$\hat{A} - \lambda \hat{B} = Q^{*}(A - \lambda B)Z$ implies
\begin{displaymath}
F\hat{A}^{T}F - \lambda F\hat{B}^{T}F = (FZ^{T}F)(FA^{T}F - \lambda FB^{T}F)(F\overline{Q}F), 
\end{displaymath}
which shows that the roles of $Q$ and $Z$ are reversed in the flipped procedure.  (Of course $Q$ and 
$F \overline{Q} F$ are not exactly the same, but they contain the same information.)  The ``compute $Q$ first'' 
procedure that we have just outlined is a way of implementing the ``flipped'' procedure without actually 
doing the flips.  

\subsection*{Backward error analysis}
It suffices to prove backward stability in Case 1, since the options in Case 2 are both variants of Case 1.

The swapping operation is a unitary equivalence, and such transformations generally are stable \cite{Hig02}, 
but there is one thing we have to check.   The core $Q$ is designed so that $Q^{*}(BZ)$ has a zero in the $(2,1)$
position.  This automatically creates a zero in the $(2,1)$ position of $Q^{*}(AZ)$ because the first columns of 
$AZ$ and $BZ$ are both proportional to $y$.  This is true in exact arithmetic.  We just need to check that in floating-point 
arithmetic the entry that is created in the $(2,1)$ position of $Q^{*}AZ$ is small enough that backward stability is
not compromised by setting it to zero.  For this it suffices that its magnitude be no bigger than a modest multiple 
of $u\norm{A}$, where $u$ is the unit roundoff.  

The swapping operation begins with the computation of $x$ in (\ref{eq:xdef}).  In floating-point arithmetic we get 
\begin{equation}\label{eq:flxdef}
\fl{x} = \left[\begin{array}{c} \alpha_{2}b(1 + \epsilon_{1}) - \beta_{2}a(1 + \epsilon_{2}) \\
\beta_{2}\alpha_{1}(1 + \epsilon_{3}) - \alpha_{2}\beta_{1}(1 + \epsilon_{4})
\end{array}\right],
\end{equation}
where each $\epsilon_{i}$ is the result of two roundoff errors,  a multiplication and
a subtraction mapped back to the product terms, and therefore satisfies 
$\absval{\epsilon_{i}} \leq 2u + O(u^{2})$.  We will use the abbreviation $\absval{\epsilon_{i}} \lesssim u$ to mean that 
$\absval{\epsilon_{i}}$ is no bigger than a modest constant times $u$.  

The next step is to compute $Z$.  In practice we do this using $\fl{x}$ and make additional roundoff errors in
the computation.  We get $\tilde{Z} = \fl{Z}$ satisfying 
\begin{equation}\label{eq:xtildedef}
\tilde{Z}e_{1} = \tilde{x} = \tilde{\gamma}^{-1}\left[\begin{array}{c} \fl{x_{1}}(1 + \epsilon_{5}) \\
\fl{x_{2}}(1 + \epsilon_{6})\end{array}\right].
\end{equation}
Here $\tilde{\gamma} = \norm{\fl{x}}$.  A tiny relative error is made during this norm computation, and another tiny error 
is made when $\fl{x_{1}}$ is divided by $\tilde{\gamma}$.  These are the causes of the error $\epsilon_{5}$, and we have 
$\absval{\epsilon_{5}} \lesssim u$.  Similarly $\absval{\epsilon_{6}} \lesssim u$.  For more details about this 
computation see \cite[\S~1.4]{AuMaRoVaWa18}.

The vector $\tilde{x}$ defined by (\ref{eq:xtildedef}) is our computed (and normalized) version of a right eigenvector
associated with eigenvalue $\sigma_{2}$.  For later use we wish to show that $\tilde{x}$ is exactly an 
eigenvector of a slightly perturbed pencil.  Thus we seek $\tilde{\alpha}_{1}$, $\tilde{\alpha}_{2}$, $\tilde{\beta}_{1}$,
and $\tilde{\beta}_{2}$ such that
\begin{equation}\label{eq:perteigvec}
\left( \tilde{\beta}_{2}\left[\begin{array}{cc} \tilde{\alpha}_{1} & a \\ & \tilde{\alpha}_{2} \end{array}\right] - 
\tilde{\alpha}_{2}\left[\begin{array}{cc} \tilde{\beta}_{1} & b \\ & \tilde{\beta}_{2} \end{array}\right] \right)
\left[\begin{array}{c} \tilde{x}_{1} \\ \tilde{x}_{2}\end{array}\right] = 
\left[\begin{array}{c} 0 \\ 0 \end{array}\right].
\end{equation}
Notice that we are not going to back any of the error onto $a$ or $b$.   This equation is equivalent to 
\begin{displaymath}
(\tilde{\beta}_{2}\tilde{\alpha}_{1} - \tilde{\alpha}_{2}\tilde{\beta}_{1})\tilde{x}_{1} +
(\tilde{\beta}_{2}a - \tilde{\alpha}_{2}b)\tilde{x}_{2} = 0.
\end{displaymath} 
Filling in the values of $\tilde{x}_{1}$ and $\tilde{x}_{2}$ from (\ref{eq:xtildedef}) and (\ref{eq:flxdef}), we 
can check that this equation holds if we make the assignments
\begin{displaymath}
\tilde{\alpha}_{1} = \alpha_{1}\frac{(1+\epsilon_{3})(1 + \epsilon_{6})}{(1+ \epsilon_{2})(1 + \epsilon_{5})}, \qquad
\tilde{\alpha}_{2} = \alpha_{2}(1 + \epsilon_{1})(1 + \epsilon_{5}),
\end{displaymath}
\begin{displaymath}
\tilde{\beta}_{1} = \beta_{1}\frac{(1 + \epsilon_{4})(1 + \epsilon_{6})}{(1 + \epsilon_{1})(1 + \epsilon_{5})}, \qquad 
\tilde{\beta}_{2} = \beta_{2}(1 + \epsilon_{2})(1 + \epsilon_{5}).
\end{displaymath}
Clearly $\absval{\tilde{\alpha}_{i} - \alpha_{i}} \lesssim u \absval{\alpha_{i}}$ and 
$\absval{\tilde{\beta}_{i} - \beta_{i}} \lesssim u \absval{\beta_{i}}$ for $i = 1$, $2$.  Equation (\ref{eq:perteigvec}) 
can be written more compactly as 
\begin{equation}\label{eq:perteigcompact}
\tilde{\beta}_{2} \tilde{A}\tilde{x} = \tilde{\alpha}_{2}\tilde{B}\tilde{x}.
\end{equation}
Thus $\tilde{x}$ is an eigenvector of the perturbed pencil $\tilde{A} - \lambda \tilde{B}$ associated with eigenvalue
$\tilde{\sigma}_{2} = \tilde{\alpha}_{2}/\tilde{\beta}_{2}$.
We also write 
\begin{equation}\label{eq:tildeperta}
\tilde{A} = A + \delta A \quad\mbox{and}\quad \tilde{B} = B + \delta B_{1},
\end{equation}
with $\delta A$ and $\delta B_{1}$ diagonal matrices satisfying $\norm{\delta A} \lesssim u \norm{A}$ and 
$\norm{\delta B_{1}} \lesssim u \norm{B}$.

Finally we compute $Q$.  In exact arithmetic $Q$ is constructed so that $Q^{*}(BZe_{1}) = \eta\,e_{1}$, for some $\eta$, 
so the first column of $Q$ must be proportional to $BZe_{1}$.   In practice, instead of $BZe_{1}$ we use
\begin{displaymath}
\check{y} = \fl{B\tilde{Z}e_{1}} = \fl{B\tilde{x}} = 
\tilde{\gamma}^{-1}\left[\begin{array}{c} \beta_{1}\tilde{x}_{1}(1+ \epsilon_{1}') + 
b \tilde{x}_{2}(1 + \epsilon_{2}') \\ \beta_{2}\tilde{x}_{2}(1 + \epsilon_{3}') \end{array}\right],
\end{displaymath}
where $\absval{\epsilon_{i}'} \lesssim u$ for $i=1$, $2$, $3$.  The computed version of $Q$ is 
$\tilde{Q} =\fl{Q}$ satisfying 
\begin{displaymath}
\tilde{Q}e_{1} = \check{\zeta}^{-1}\left[\begin{array}{c} 
\check{y}_{1}(1 + \epsilon_{4}')  \\ \check{y}_{2}(1 + \epsilon_{5}')
\end{array}\right],
\end{displaymath}
where $\check{\zeta} = \norm{\check{y}}$, and $\epsilon_{4}'$ and $\epsilon_{5}'$ are due to the tiny roundoff
errors in the calculation.  

For our analysis we need to establish that there is a slightly perturbed matrix 
\begin{displaymath}
\hat{B} = B + \delta B_{2} = \left[\begin{array}{cc} \hat{\beta}_{1} & b \\ & \hat{\beta}_{2} \end{array}\right]
\end{displaymath}
such that $\tilde{Q}^{*}\hat{B}\tilde{Z}$ has an exact zero in the $(2,1)$ position.  This just means that 
$\tilde{y} = \tilde{Q}e_{1}$  is exactly proportional to $\hat{B}\tilde{Z}e_{1} = \hat{B}\tilde{x}$.  It is easy 
to check that the choice 
\begin{displaymath}
\hat{\beta}_{1} = \beta_{1}\frac{(1 + \epsilon_{1}')}{(1 + \epsilon_{2}')}, \qquad 
\hat{\beta}_{2} = \beta_{2}\frac{(1+\epsilon_{3}')(1 + \epsilon_{5}')}{(1 + \epsilon_{2}')(1 + \epsilon_{4}')}
\end{displaymath}
does the trick.  Clearly $\absval{\hat{\beta}_{1} - \beta_{1}} \lesssim u\,\absval{\beta_{1}}$ and
$\absval{\hat{\beta}_{2} - \beta_{2}} \lesssim u\,\absval{\beta_{2}}$, and $\delta B_{2}$ is a diagonal 
matrix satisfying $\norm{\delta B_{2}} \lesssim u\, \norm{B}$.

Our final computed results are $\fl{\tilde{Q}^{*}A\tilde{Z}}$ and $\fl{\tilde{Q}^{*}B\tilde{Z}}$.  We have to show
that the $(2,1)$ entries of these matrices are small enough that we can set them to zero without compromising 
backward  stability.  The ``$B$'' part is routine.  Focusing on the $(2,1)$ entry, we have 
\begin{displaymath}
e_{2}^{T}\fl{\tilde{Q}^{*}B\tilde{Z}}e_{1} =  e_{2}^{T}\tilde{Q}^{*}B\tilde{Z}e_{1} + e_{2}^{T}E_{1}e_{1},    
\end{displaymath}
where $E_{1}$ is the matrix of roundoff errors incurred in multiplying the three matrices together and 
satisfies $\norm{E_{1}} \lesssim u\,\norm{\tilde{Q}}\,\norm{B}\,\norm{\tilde{Z}}$, i.e.\ $\norm{E_{1}} \lesssim u\,\norm{B}$.  
The remaining term is 
\begin{displaymath}
e_{2}^{T}\tilde{Q}^{*}B\tilde{Z}e_{1} = e_{2}^{T}\tilde{Q}^{*}\hat{B}\tilde{Z}e_{1} - e_{2}^{T}\tilde{Q}^{*}\delta B_{2}\tilde{Z}e_{1}.
\end{displaymath}
The first term on the right-hand side is exactly zero by construction.  The second is bounded above by 
$\norm{\delta B_{2}} \lesssim u\,\norm{B}$.  This takes care of the ``$B$'' part.

The ``$A$'' part (the important part) is more delicate.  We have 
\begin{displaymath}
e_{2}^{T}\fl{\tilde{Q}^{*}A\tilde{Z}}e_{1} =  e_{2}^{T}\tilde{Q}^{*}A\tilde{Z}e_{1} + e_{2}^{T}E_{2}e_{1},    
\end{displaymath}
where $E_{2}$ is the matrix of roundoff errors incurred in multiplying the three matrices together
and satisfies $\norm{E_{2}} \lesssim u\,\norm{A}$.  
The remaining term is 
\begin{displaymath}
e_{2}^{T}\tilde{Q}^{*}A\tilde{Z}e_{1} = e_{2}^{T}\tilde{Q}^{*}\tilde{A}\tilde{Z}e_{1} - e_{2}^{T}\tilde{Q}^{*}\delta A\tilde{Z}e_{1}.
\end{displaymath}
The second term on the right-hand side is bounded above by $\norm{\delta A} \lesssim u\,\norm{A}$, so now we can just focus on the other term.  Here we make use of (\ref{eq:perteigcompact}), which can be written as 
$\tilde{A}\tilde{Z}e_{1} = (\tilde{\alpha}_{2}/\tilde{\beta}_{2})\tilde{B}\tilde{Z}e_{1}$.  

\begin{displaymath}
e_{2}^{T}\tilde{Q}^{*}\tilde{A}\tilde{Z}e_{1} = 
\frac{\tilde{\alpha}_{2}}{\tilde{\beta}_{2}}e_{2}^{T}\tilde{Q}^{*}\tilde{B}\tilde{Z}e_{1} =
\frac{\tilde{\alpha}_{2}}{\tilde{\beta}_{2}}e_{2}^{T}\tilde{Q}^{*}\hat{B}\tilde{Z}e_{1} +
\frac{\tilde{\alpha}_{2}}{\tilde{\beta}_{2}}e_{2}^{T}\tilde{Q}^{*}(\delta B_{1} - \delta B_{2})\tilde{Z}e_{1}.
\end{displaymath}
The term containing $\hat{B}$ is zero by construction, so now we just need to concentrate on the other term.
Let $\delta B = \delta B_{1} - \delta B_{2}$.  From the definitions of $\delta B_{1}$ and $\delta B_{2}$ we see that
\begin{displaymath}
\delta B = \left[\begin{array}{cc} \epsilon_{1}''\,\beta_{1} & 0 \\ 0 & \epsilon_{2}''\,\beta_{2} \end{array}\right], 
\end{displaymath}
where $\absval{\epsilon_{i}''} \lesssim u$ for $i=1$, $2$.  Moreover 
$\displaystyle\frac{\tilde{\alpha}_{2}}{\tilde{\beta}_{2}} = \frac{\alpha_{2}}{\beta_{2}}(1 + \epsilon_{3}'')$ for some 
tiny $\epsilon_{3}''$.   We also use our assumption $\absval{\sigma_{1}} \geq \absval{\sigma_{2}}$ to deduce that
$\absval{\beta_{1}\alpha_{2}/\beta_{2}} \leq \absval{\alpha_{1}}$.
Thus
\begin{displaymath}
\absval{(\tilde{\alpha}_{2}/\tilde{\beta}_{2})\delta B} = (1 + \epsilon_{3}'') 
\left[\begin{array}{cc} \absval{\epsilon_{1}''\,\beta_{1}\alpha_{2}/\beta_{2}} & 
\\ & \absval{\epsilon_{2}''\,\alpha_{2}} \end{array}\right] \leq 
(1 + \epsilon_{3}'') 
\left[\begin{array}{cc} \absval{\epsilon_{1}'' \,\alpha_{1}} & 
\\ & \absval{\epsilon_{2}''\,\alpha_{2}} \end{array}\right],
\end{displaymath}
so 
\begin{displaymath}
\norm{(\tilde{\alpha}_{2}/\tilde{\beta}_{2})\delta B} \lesssim u\,\norm{A}. 
\end{displaymath}
We conclude that 
our one remaining term, which is $(\tilde{\alpha}_{2}/\tilde{\beta}_{2})e_{2}^{T}\tilde{Q}^{*}(\delta B)\tilde{Z}e_{1}$, 
satisfies 
\begin{displaymath}
\absval{(\tilde{\alpha}_{2}/\tilde{\beta}_{2})e_{2}^{T}\tilde{Q}^{*}(\delta B)\tilde{Z}e_{1}} \lesssim u\,\norm{A}.
\end{displaymath}

We have demonstrated that 
\begin{displaymath}
\absval{e_{2}^{T}\fl{\tilde{Q}^{*}A\tilde{Z}}e_{1}} \lesssim u\,\norm{A} \quad \mbox{and} \quad
\absval{e_{2}^{T}\fl{\tilde{Q}^{*}B\tilde{Z}}e_{1}} \lesssim u\,\norm{B}, 
\end{displaymath}
so we can set these numbers
to zero without compromising backward stability.  The $\lesssim$ symbols hide constants, but these constants are not too large due to the small total number of operations required by the swap.

Our procedure improves on that of Van Dooren \cite{VanD81} in that the latter only guarantees that the 
two entries are bounded above by $u\max\{\norm{A},\norm{B}\}$ instead of $u\,\norm{A}$ and $u\,\norm{B}$ separately.
It follows that our procedure produces better results in cases where $A$ and $B$ have vastly different norms.  
We remind the reader that the $A$ and $B$ referred to here are the small matrices defined in (\ref{eq:startsubpencil}), 
and not the larger matrices in which they are imbedded.   Therefore we cannot solve the problem of different norms
by a simple rescaling of the large matrices at the outset, as that does not guarantee equal norms in all of the little submatrices in which the swaps take place.

\subsection*{Numerical experiments}

In most cases it does not matter which swapping procedure is used; they all perform well.  
In order to see a difference, they must be stress tested on pencils that have elements that vary 
widely in magnitude.   Therefore, in the two experiments reported here, we used pencils whose nonzero entries
are randomly generated complex numbers with magnitudes distributed logarithmically in the range from $10^{-12}$ to $10^{12}$. 

In our first test we generated sixty-four million random $2 \times 2$ upper-triangular 
pencils and computed the swapping transformations using three different algorithms:
our method, the method of Van Dooren \cite{VanD81}, and a method that solves the generalized Sylvester
equation explicitly to determine $Q$ and $Z$ \cite{BojVan93}.  The computations were done in IEEE standard 
double-precision arithmetic, for which $u \approx 10^{-16}$.   Table~\ref{tab:1} shows that our method 
always produces residuals $\absval{a_{21}}/\norm{A}$ and $\absval{b_{21}}/\norm{B}$ that are under $10^{-15}$, 
and more than $99.7\%$ of them are under $10^{-16}$.   In contrast, the Van Dooren and Sylvester methods 
sometimes produce much larger residuals, approaching $10^{0}$ in a few cases.  If we change the criterion and 
consider the residuals $\absval{a_{21}}/\Delta$ and $\absval{b_{21}}/\Delta$, where 
$\Delta = \max\{\norm{A},\norm{B}\}$, then all methods perform well, as Table~\ref{tab:2} shows.  By this 
criterion all residuals are under $10^{-15}$.  Our method and Van Dooren's method perform about 
equally well, and the Sylvester method is almost as good.   We conclude that if $\norm{A}$ and $\norm{B}$ are
roughly the same, it doesn't matter which method is used.  However, in problems for which there can be large
differences in magnitude between $\norm{A}$ and $\norm{B}$, our method is better. 

\begin{table}
  \centering
  \caption{Distribution of errors $\absval{\hat{a}_{21}}/\|A\|$ and $\absval{\hat{b}_{21}}/\|B\|$ for our method, 
    Van Dooren's method, and the Sylvester method.}
  \begin{tabular*}{\textwidth}{@{\extracolsep{\fill}}cC{1em}cC{0.13\textwidth}C{0.13\textwidth}C{0.13\textwidth}C{0.13\textwidth}}
    \toprule
    $\absval{\hat{x}_{21}}/\|X\|$& & {\small $\left[0, 10^{-16}\right]$} & {\small $\left(10^{-16}, 10^{-15}\right]$} & 
    {\small $\left(10^{-15}, 10^{-10}\right]$} & {\small $\left(10^{-10}, 10^{-5}\right]$} & {\small $\left(10^{-5},10^0 \right]$}\\
    \midrule
    \multirow{2}{*}{Our method}& $A$ & $99.71\%$ & $0.29\%$ & $0\%$ & $0\%$ & $0\%$ \\
                               & $B$ & $99.85\%$ & $0.15\%$ & $0\%$ & $0\%$ & $0\%$ \\[1.2ex]

    \multirow{2}{*}{Van Dooren}& $A$ & $98.19\%$ & $0.55\%$ & $0.93\%$ & $0.27\%$ & $0.06\%$ \\
                               & $B$ & $98.19\%$ & $0.55\%$ & $0.93\%$ & $0.27\%$ & $0.06\%$ \\[1.2ex]
   
    \multirow{2}{*}{Sylvester}& $A$ & $93.34\%$ & $5.88\%$ & $0.57\%$ & $0.17\%$ & $0.04\%$ \\
                              & $B$ &  $93.34\%$ & $5.88\%$ & $0.57\%$ & $0.17\%$ & $0.04\%$ \\                                                    
    \bottomrule
  \end{tabular*}
  \label{tab:1}
\end{table}

\begin{table}
\centering
\caption{Distribution of errors $\absval{\hat{a}_{21}}/\Delta$ and $\absval{\hat{b}_{21}}/\Delta$ for our method, 
Van Dooren's method, and the Sylvester method.}
  \begin{tabular*}{\textwidth}{@{\extracolsep{\fill}}cC{1em}cC{0.13\textwidth}C{0.13\textwidth}C{0.13\textwidth}C{0.13\textwidth}}
\toprule
$\absval{\hat{x}_{21}}/\Delta$&  & {\small$\left[0, 10^{-16}\right]$} & {\small$\left(10^{-16}, 10^{-15}\right]$} & {\small$\left(10^{-15}, 10^{-10}\right]$} & {\small$\left(10^{-10}, 10^{-5}\right]$} & {\small$\left(10^{-5},10^0 \right]$}\\
\midrule
\multirow{2}{*}{Our method}& $A$ & $99.87\%$ & $0.13\%$ & $0\%$ & $0\%$ & $0\%$ \\
                           & $B$ & $99.93\%$ & $0.07\%$ & $0\%$ & $0\%$ & $0\%$ \\[1.2ex]

\multirow{2}{*}{Van Dooren}& $A$ & $99.94\%$ & $0.06\%$ & $0\%$ & $0\%$ & $0\%$ \\
                           & $B$ & $99.94\%$ & $0.06\%$ & $0\%$ & $0\%$ & $0\%$ \\[1.2ex]

\multirow{2}{*}{Sylvester}& $A$ & $97.26\%$ & $2.74\%$ & $0\%$ & $0\%$ & $0\%$ \\
                           & $B$ &  $97.26\%$ & $2.74\%$ & $0\%$ & $0\%$ & $0\%$ \\                                                    
\bottomrule
\end{tabular*}
\label{tab:2}
\end{table}

It is natural to ask whether improved backward stability of the swapping transformations actually results 
in more accurate computed eigenvalues of the larger pencils.  To test this we considered ten thousand randomly generated
$3 \times 3$ upper Hessenberg pencils with logarithmically distributed entries with magnitudes varying from $10^{-12}$
to $10^{12}$.  Since we do not know the exact eigenvalues of these pencils, we used MATLAB with the ADVANPIX 
Multiprecision Computing Toolbox\footnote{https://www.advanpix.com} to compute ``exact'' eigenvalues in quadruple 
precision arithmetic.  We compared these with the approximate eigenvalues computed using our method and 
Van Dooren's.  

Before we look at that comparison, we note we didn't just compute the eigenvalues;   
in fact we computed the Schur form 
$A_{T} - \lambda B_{T} = Q^{*}(A - \lambda B)Z$, where $A_{T}$ and $B_{T}$ are upper triangular.  
This allowed us to compute residuals
\begin{equation}\label{eq:schurbackerrs}
r_{A} = \norm{A - QA_{T}Z^{*}}/\norm{A} \quad\mbox{and}\quad r_{B} = \norm{B - QB_{T}Z^{*}}/\norm{B},
\end{equation}
which are measures of backward error.  When the computation was done using our method, the residuals 
were always tiny, never exceeding $10^{-14}$, verifying normwise backward stability.  When Van Dooren's 
criterion was used, the residuals (\ref{eq:schurbackerrs}) were usually just as small but occasionally 
larger.    If the denominators $\norm{A}$ and $\norm{B}$ in the residuals $r_{A}$ and $r_{B}$ are replaced by 
$\Delta = \max\{\norm{A},\norm{B}\}$,  
the Van Dooren residuals also become uniformly small, never exceeding $10^{-14}$.

Of course tiny backward errors do not guarantee accurate computed eigenvalues, as some of them may be 
ill conditioned.  Moreover, decreasing the backward error does not necessarily guarantee improved eigenvalue
accuracy, so we must make the comparison.  
Let $\lambda_{i}$, $i=1$, $2$, $3$, denote the ``exact'' eigenvalues produced in quadruple precision, 
let $\lambda_{i}^{(o)}$ denote the approximate eigenvalues computed by our method, and let 
\begin{equation}\label{eq:relerrdef}
e^{(o)}   =  \max_{i}\absval{\lambda_{i}^{(o)} - \lambda_{i}}/\absval{\lambda_{i}},
\end{equation}
the maximum relative error.  Let $\lambda_{i}^{(v)}$ denote the eigenvalues computed by Van Dooren's method, 
and let $e^{(v)}$ denote the maximum relative error, defined analogously to $e^{(o)}$ as in (\ref{eq:relerrdef}).

We examined the ratios $e^{(v)}/e^{(0)}$ and found that just over 98\% of our trials resulted in  $0.1 < e^{(v)}/e^{(o)} < 10$, indicating that neither method was significantly more accurate than the other.   (In fact there were many cases where 
$e^{(v)}/e^{(o)}=1$, since it often happens that our criterion and Van Dooren's criterion make exactly the same decisions.)  
Of the remaining trials, which numbered 181, there were 145 in which our method did significantly better than Van Dooren's, 
i.e.\  $e^{(v)}/e^{(o)} > 10$, and 36 in which $e^{(v)}/e^{(o)} < 0.1$.   Thus our new method obtained more
accurate eigenvalues in about 80\% of the significant cases.  More details are given in histogram form in 
Figure~\ref{hist:ratplot}.   The gap in the center of the figure 
is due to having left out the many cases for which $e^{(v)}/e^{(o)}$ is close
to $1$.   In the interest of compactness and clarity, the figure also leaves out one ``off the charts'' case 
for which $e^{(v)}/e^{(o)} \approx 10^{15}$.
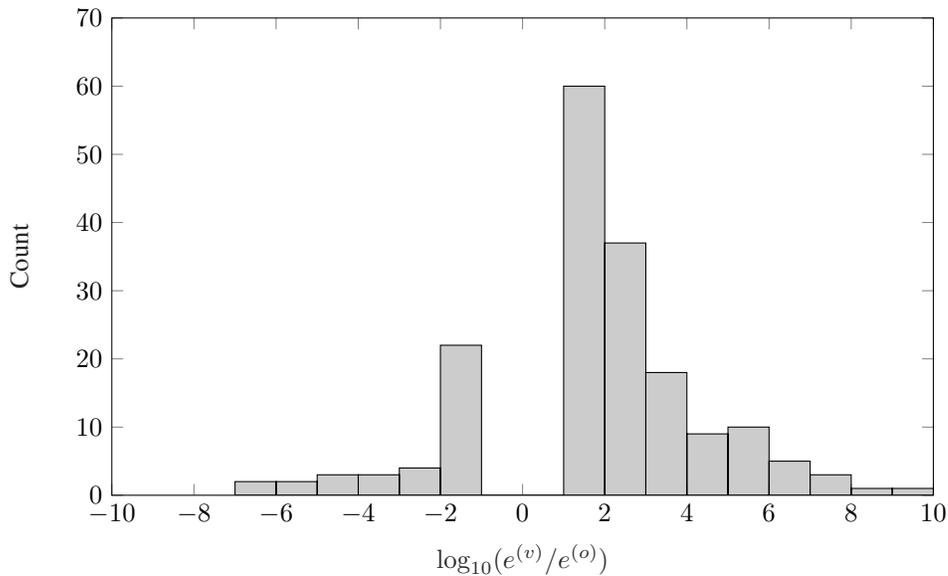
\begin{figure}
\caption{Histogram of logarithm of  $e^{(v)}/e^{(0)}$ in significant cases}
\begin{center}
\begin{tikzpicture}
\begin{axis}[%
width = 4.3in,
height = 2.5in,
at={(1.011in,0.642in)},
scale only axis,
xmin=-10,
xmax=10,
xlabel style={font=\color{white!15!black}},
xlabel={$\log_{10}(e^{(v)}/e^{(o)})$},
ymin=0,
ymax=70,
ylabel style={font=\color{white!15!black}},
ylabel={Count},
axis background/.style={fill=white},
]
\addplot[ybar interval, fill=gray, fill opacity=0.4, draw=black, area legend] table[row sep=crcr] {%
x	y\\
-7    2 \\
-6    2 \\ 
-5    3 \\
-4    3 \\
-3    4 \\
-2   22 \\
-1    0 \\
0     0 \\
1   60 \\
2   37 \\
3   18 \\
4    9  \\
5   10 \\
6    5  \\
7    3  \\
8    1  \\
9    1  \\
10  0 \\
};
\end{axis}
\end{tikzpicture}%
\end{center}
\label{hist:ratplot}
\end{figure}

\section{Conclusions}

We have discussed the RQZ algorithm and a number of variants, which we refer to generally as pole-swapping algorithms.
We have made two main contributions:  1) We have developed a flexible, modular convergence theory that can be
applied to any pole-swapping algorithm.  2) We have presented a new, more accurate, swapping procedure.  A backward
error analysis and numerical experiments demonstrate the superiority of the new procedure. 

\section*{Acknowledgment}  We thank the anonymous referees for carefully reading the paper and suggesting several improvements.


\begin{thebibliography}{10}

\bibitem{AuMaRoVaWa18}
{\sc J.~L. Aurentz, T.~Mach, L.~Robol, R.~Vandebril, and D.~S. Watkins}, {\em
  Core-Chasing Algorithms for the Eigenvalue Problem}, SIAM, Philadelphia,
  2018.

\bibitem{AuMaRoVaWa18g}
\leavevmode\vrule height 2pt depth -1.6pt width 23pt, {\em Fast and backward
  stable computation of roots of polynomials, part {II}: backward error
  analysis; companion matrix and companion pencil}, {SIAM} J.\ Matrix Anal.\
  Appl., 39 (2018), pp.~1245--1269.

\bibitem{AuMaRoVaWa19}
\leavevmode\vrule height 2pt depth -1.6pt width 23pt, {\em Fast and backward
  stable computation of the eigenvalues and eigenvectors of matrix
  polynomials}, Math.\ Comp., 88 (2019), pp.~313--347.

\bibitem{AuMaVaWa15}
{\sc J.~L. Aurentz, T.~Mach, R.~Vandebril, and D.~S. Watkins}, {\em Fast and
  backward stable computation of roots of polynomials}, SIAM J. Matrix Anal.
  Appl., 36 (2015), pp.~942--973.

\bibitem{BaiDem93}
{\sc Z.~Bai and J.~Demmel}, {\em On swapping diagonal blocks in real {S}chur
  form}, Linear Algebra Appl., 186 (1993), pp.~73--95.

\bibitem{BerGut15}
{\sc M.~Berljafa and S.~G\"uttel}, {\em Generalized rational {K}rylov
  decompositions with an application to rational approximation}, SIAM J.\
  Matrix Anal.\ Appl., 36 (2015), pp.~894--916.

\bibitem{BojVan93}
{\sc A.~{B}ojanczyk and P.~{V}an {D}ooren}, {\em Reordering diagonal blocks in
  the real {S}chur form}, in Linear Algebra for Large Scale and Real-Time
  Applications, M.~Moonen, G.~Golub, and B.~D. Moor, eds., {NATO ASI} {S}eries
  {E}: Applied Sciences, Springer, 1993, pp.~351--352.

\bibitem{p522}
{\sc K.~Braman, R.~Byers, and R.~Mathias}, {\em The multishift {QR} algorithm.
  part {I}: Maintaining well-focused shifts and level 3 performance}, {SIAM}
  J.\ Matrix Anal.\ Appl., 23 (2002), pp.~929--947.

\bibitem{BrByMa01}
{\sc K.~Braman, R.~Byers, and R.~Matthias}, {\em The multishift {QR} algorithm,
  part {I}: Maintaining well focused shifts and level 3 performance}, {SIAM}
  J.\ Matrix Anal.\ Appl., 23 (2001), pp.~929--947.

\bibitem{BrByMa01b}
\leavevmode\vrule height 2pt depth -1.6pt width 23pt, {\em The multishift {QR}
  algorithm, part {II}: Aggressive early deflation}, {SIAM} J.\ Matrix Anal.\
  Appl., 23 (2001), pp.~948--973.

\bibitem{Bye83}
{\sc R.~Byers}, {\em Hamiltonian and Symplectic Algorithms for the Algebraic
  {R}iccati Equation}, PhD thesis, Cornell University, 1983.

\bibitem{Bye86}
\leavevmode\vrule height 2pt depth -1.6pt width 23pt, {\em A {H}amiltonian {QR}
  algorithm}, SIAM J.\ Sci.\ Stat.\ Comput., 7 (1986), pp.~212--229.

\bibitem{Cam19}
{\sc D.~Camps}, {\em Pole swapping methods for the eigenvalue problem:
  {R}ational {QR} algorithms}, PhD thesis, {KU} {L}euven, 2019.

\bibitem{CaMeVa19a}
{\sc D.~Camps, K.~Meerbergen, and R.~Vandebril}, {\em A rational {QZ} method},
  SIAM J.\ Matrix Anal.\ Appl., 40 (2019), pp.~943--972.

\bibitem{Fra61b}
{\sc J.~G.~F. Francis}, {\em The {QR} transformation, part {II}}, Computer J.,
  4 (1961), pp.~332--345.

\bibitem{GolVan13}
{\sc G.~H. Golub and C.~F. {V}an {L}oan}, {\em Matrix Computations}, {J}ohns
  {H}opkins {U}niversity {P}ress, {B}altimore, {F}ourth~ed., 2013.

\bibitem{GrKaKr10}
{\sc R.~Granat, B.~K{\aa}gstr\"{o}m, and D.~Kressner}, {\em A novel parallel
  {QR} algorithm for hybrid distributed memory {HPC} systems}, {SIAM} J.\ Sci.\
  Comput., 32 (2010), pp.~2345--2378.

\bibitem{Hig02}
{\sc N.~J. Higham}, {\em Accuracy and Stability of Numerical Algorithms}, SIAM,
  Philadelphia, 2nd~ed., 2002.

\bibitem{KagPor96a}
{\sc B.~K{\aa}gstr\"{o}m and P.~Poromaa}, {\em Computing eigenspaces with
  specified eigenvalues of a regular matrix pair {(A, B)} and condition
  estimation: Theory, algorithms and software}, Numerical Algorithms, 12
  (1996), pp.~369--407.

\bibitem{Kagpor96b}
\leavevmode\vrule height 2pt depth -1.6pt width 23pt, {\em Lapack-style
  algorithms and software for solving the generalized sylvester equation and
  estimating the separation between regular matrix pairs}, ACM Trans. Math.
  Softw., 22 (1996), pp.~78--103.

\bibitem{KaKrLa14}
{\sc L.~Karlsson, D.~Kressner, and B.~Lang}, {\em Optimally packed chains of
  bulges in multishift {QR} algorithms}, {ACM} Trans.\ Math.\ Software, 40
  (2014).

\bibitem{KrScWa08}
{\sc D.~Kressner, C.~Schr\"oder, and D.~S. Watkins}, {\em Implicit {QR}
  algorithms for palindromic and even eigenvalue problems}, Numer. Algorithms,
  51 (2009), pp.~209--238.

\bibitem{Lang97}
{\sc B.~Lang}, {\em Effiziente Orthogonaltransformationen bei der Eigen- und
  Singul\"{a}rwertzerlegung.}, Habilitationsschrift, Universit\"{a}t Wuppertal,
  Wuppertal, Germany, 1997.

\bibitem{Lang98}
\leavevmode\vrule height 2pt depth -1.6pt width 23pt, {\em Using level 3 {BLAS}
  in rotation-based algorithms}, {SIAM} J.\ Sci.\ Comput., 19 (1998),
  pp.~626--634.

\bibitem{MolSte73}
{\sc C.~B. Moler and G.~W. Stewart}, {\em An algorithm for generalized matrix
  eigenvalue problems}, {SIAM} J.\ Numer.\ Anal., 10 (1973), pp.~241--256.

\bibitem{CaMeVa19b}
{\sc T.~Steel, D.~Camps, K.~Meerbergen, and R.~Vandebril}, {\em A multishift,
  multipole rational {QZ} method with aggressive early deflation}.
\newblock arXiv:1902.10954, 2020.
\newblock submitted for publication.

\bibitem{VanD81}
{\sc P.~Van~Dooren}, {\em A generalized eigenvalue approach for solving
  {R}iccati equations}, {SIAM} J.\ Sci.\ Stat.\ Comput., 2 (1981),
  pp.~121--135.

\bibitem{VanWat12g}
{\sc R.~Vandebril and D.~S. Watkins}, {\em An extension of the {QZ} algorithm
  beyond the {Hessenberg}-upper triangular pencil}, Electron.\ Trans.\ Numer.\
  Anal., 40 (2012), pp.~17--35.

\bibitem{Wat95}
{\sc D.~S. Watkins}, {\em Forward stability and transmission of shifts in the
  {QR} algorithm}, SIAM J.\ Matrix Anal.\ Appl., 16 (1995), pp.~469--487.

\bibitem{Wat96}
\leavevmode\vrule height 2pt depth -1.6pt width 23pt, {\em The transmission of
  shifts and shift blurring in the {QR} algorithm}, Linear Algebra Appl.,
  241--243 (1996), pp.~877--896.

\bibitem{Wat98}
\leavevmode\vrule height 2pt depth -1.6pt width 23pt, {\em Bulge exchanges in
  algorithms of {QR} type}, SIAM J.\ Matrix Anal.\ Appl., 19 (1998),
  pp.~1074--1096.

\bibitem{Wat07}
\leavevmode\vrule height 2pt depth -1.6pt width 23pt, {\em The Matrix
  Eigenvalue Problem: {GR} and {Krylov} Subspace Methods}, SIAM, Philadelphia,
  2007.

\bibitem{Wat10}
\leavevmode\vrule height 2pt depth -1.6pt width 23pt, {\em Fundamentals of
  Matrix Computations}, Wiley, New York, 3rd~ed., 2010.

\bibitem{Wat11}
\leavevmode\vrule height 2pt depth -1.6pt width 23pt, {\em Francis's
  algorithm}, Amer.\ Math.\ Monthly, 118 (2011), pp.~387--403.

\bibitem{Wil65}
{\sc J.~H. Wilkinson}, {\em The {A}lgebraic {E}igenvalue {P}roblem}, Clarendon
  Press, Oxford University, 1965.

\end{thebibliography}

\end{document}